\documentclass[11pt,reqno]{amsart}

\usepackage{amssymb, amsmath, amsthm}
\usepackage{hyperref}
\usepackage[alphabetic,lite]{amsrefs}
\usepackage{verbatim}
\usepackage{amscd}   
\usepackage[all]{xy} 
\usepackage{youngtab} 
\usepackage{young} 
\usepackage{ytableau}
\usepackage{tikz}
\usepackage{ mathrsfs }

\renewcommand{\1}{\underline{1}}
\renewcommand{\a}{\alpha}

\newcommand{\A}{\underline{A}}
\newcommand{\B}{\underline{B}}
\renewcommand{\b}{\beta}
\renewcommand{\c}{\gamma}
\renewcommand{\d}{\underline{d}}

\newcommand{\K}{\bb{K}}
\renewcommand{\ll}{\lambda}
\newcommand{\m}{\mathfrak{m}}

\newcommand{\mfc}{\mathfrak{c}}

\newcommand{\N}{\underline{N}}
\renewcommand{\o}{\otimes}
\newcommand{\oo}{\otimes}

\renewcommand{\r}{\underline{r}}
\newcommand{\s}{\sigma}

\renewcommand{\SS}{\mathfrak{S}}
\renewcommand{\t}{\tau}

\newcommand{\FF}{\mathscr{F}}
\newcommand{\II}{\mathscr{I}}
\newcommand{\KK}{\mathscr{K}}
\newcommand{\SSS}{\mathscr{S}}

\newcommand{\GL}{{GL}}
\newcommand{\Hom}{\operatorname{Hom}} 
\newcommand{\Sym}{\operatorname{Sym}} 

\newcommand{\bb}[1]{\mathbb{#1}}
\newcommand{\defi}[1]{{\upshape\sffamily #1}}
\renewcommand{\rm}[1]{\mathrm{#1}}
\newcommand{\mc}[1]{\mathcal{#1}}
\newcommand{\ol}[1]{\overline{#1}}
\newcommand{\ul}[1]{\underline{#1}}
\newcommand{\tl}[1]{\tilde{#1}}
\newcommand{\scpr}[2]{\left\langle #1,#2 \right\rangle}

\newcommand{\ccircle}[1]{*+<1ex>[o][F-]{#1}}
\newcommand{\ccirc}[1]{\xymatrix@1{*+<1ex>[o][F-]{#1}}}

\newcommand{\ytabonecol}[2]{
\begin{array}{|c|}
 #1 \\
 #2 \\
 \end{array}
}

\def\lra{\longrightarrow}

\newcommand{\PP}{\mathbb{P}}

\newcommand{\CC}{\mathbb{C}}

\newtheorem{theorem}{Theorem}[section]
\newtheorem{lemma}[theorem]{Lemma}
\newtheorem{conjecture}[theorem]{Conjecture}
\newtheorem{proposition}[theorem]{Proposition}

\newtheorem*{corollary*}{Corollary}

\newcounter{tmaincounter}

\newtheorem{ttmain}[tmaincounter]{Theorem}

\theoremstyle{definition}
\newtheorem{definition}[theorem]{Definition}
\newtheorem*{definition*}{Definition}
\newtheorem{example}[theorem]{Example}

\theoremstyle{remark}
\newtheorem{remark}[theorem]{Remark}
\newtheorem*{remark*}{Remark}

\numberwithin{equation}{section}

\newcommand{\claudiu}[1]{{\color{red} \sf $\clubsuit\clubsuit\clubsuit$ Claudiu: [#1]}}
\newcommand{\luke}[1]{{\color{blue} \sf $\clubsuit\clubsuit\clubsuit$ Luke: [#1]}}

\begin{document}

\title{Tangential varieties of Segre--Veronese varieties}

\author{Luke Oeding}
\address{Department of Mathematics, University of California, Berkeley, CA 94720-3840}
\email{oeding@math.berkeley.edu}

\author{Claudiu Raicu}
\address{Department of Mathematics, Princeton University, Princeton, NJ 08544-1000\newline
\indent Institute of Mathematics ``Simion Stoilow'' of the Romanian Academy}
\email{craicu@math.princeton.edu}


\date{\today}


\begin{abstract} We determine the minimal generators of the ideal of the tangential variety of a Segre--Veronese variety, as well as the decomposition into irreducible $\GL$--representations of its homogeneous coordinate ring. In the special case of a Segre variety, our results confirm a conjecture of Landsberg and Weyman.
\end{abstract}

\maketitle

\section{Introduction}

For a projective algebraic variety $X\subset \PP^{N}$, the \defi{tangential variety} $\tau(X)$ (also known as the tangent developable or the first osculating variety of $X$) is the union of all points on all embedded tangent lines to $X$. The points in $\tau(X)$ together with those lying on the secant lines to $X$ form the (first) \defi{secant variety} of $X$, denoted $\s_2(X)$. Tangential and secant varieties were studied classically, among others by Terracini, and were brought into a modern light by F. L. Zak~\cite{Zak}. A basic problem, given any projective variety, is to understand its defining ideal and homogeneous coordinate ring. In \cite{raiGSS}, the second author solves this problem for $\s_2(X)$ when $X$ is a Segre--Veronese variety. The purpose of this paper is to solve the corresponding problem for~$\tau(X)$.

For a sequence $\d=(d_1,\cdots,d_n)$ of positive integers, the \defi{Segre--Veronese variety} is the image $X$ of the embedding
\[SV_{\d}:\bb{P}V_1^*\times\cdots\times\bb{P}V_n^*\lra\bb{P}(\Sym^{d_1}V_1^*\o\cdots\o \Sym^{d_n}V_n^*),\]
\[([e_1],\cdots,[e_n])\longmapsto [e_1^{d_1}\o\cdots\o e_n^{d_n}],\text{ for }e_i\in V_i^*\footnote{For convenience we dualize vector spaces here so that our modules of polynomials may be written without the dual.}. \]

\begin{ttmain}\label{thm:maincoordring} Let $X=SV_{\d}(\bb{P}V_1^*\times\bb{P}V_2^*\times\cdots\times\bb{P}V_n^*)$ be a Segre--Veronese variety over a field $\K$ of characteristic zero. The degree $r$ part of the homogeneous 	coordinate ring of $\tau(X)$ decomposes as
\[\K[\tau(X)]_r=\bigoplus_{\substack{\ll=(\ll^1,\cdots,\ll^n)\\ \ll^j\vdash rd_j}}(S_{\ll^1}V_1\o\cdots\o S_{\ll^n}V_n)^{m_{\ll}},\]
where $m_{\ll}$ is either $0$ or $1$, obtained as follows. Set 
\[f_{\ll}=\max_{j=1,\cdots,n}\left\lceil\frac{\ll_2^j}{d_j}\right\rceil,\quad e_{\ll}=\ll^1_2+\cdots+\ll^n_2.\]
If some $\ll^j$ has more than two parts, or $e_{\ll}<2f_{\ll}$, or $e_{\ll}>r$, then $m_{\ll}=0$, else $m_{\ll}=1$. 
\end{ttmain}
Theorem \ref{thm:maincoordring} generalizes the description of the coordinate ring of the tangential variety of a Segre variety (a special case of \cite[Theorem~5.2]{lan-wey}) and of that of a Veronese variety (see \cite[p.66]{eis} or \cite[Theorem~4.2]{lan-wey}). The next goal of our paper is to determine the generators of the ideal of $\tau(X)$. In the case when $X$ is a Segre variety, they constitute the object of the Landsberg--Weyman Conjecture:

\begin{conjecture}[{\cite[Conjecture~7.6]{lan-wey}}]\label{conj:lanwey} When $X$ is a Segre variety, $I(\tau(X))$ is generated by the submodules of quadrics which have at least four $\bigwedge^{2}$ factors, the cubics with four $S_{(2,1)}$ factors and all other factors $S_{(3)}$, the cubics with at least one $\bigwedge^3$ factor, and the quartics with three $S_{(2,2)}$'s and all other factors $S_{(4)}$.
\end{conjecture}
We call the equations appearing in the above conjecture the \defi{Landsberg--Weyman equations}. While the language of modules provides an efficient way to describe sets of equations that are invariant under a group action, in order to practically use these equations, one needs an explicit realization. This is described in \cite[Section~3.2]{oedingJPAA}. The equations of degree $2$ and $3$ in Conjecture \ref{conj:lanwey} are built from minors of matrices of flattenings, which were the main players in the case of the secant variety of the Segre variety. The quartics however are the new interesting equations for the tangential variety, constructed out of Cayley's hyperdeterminant of a $2\times 2\times 2$ tensor \cite{GKZ}. The presence of these equations lead to an unexpected connection to the variety of principal minors of symmetric matrices (studied in~\cite{oedingANT}, \cite{hol-stu}), allowing the first author to prove that a subset of the Landsberg--Weyman equations define $\tau(X)$ set--theoretically~\cite{oedingJPAA}, a weaker version of Conjecture~\ref{conj:lanwey}. Our next result confirms the strong, ideal theoretic form of the Landsberg--Weyman conjecture, and generalizes it to the Segre--Veronese case:

\begin{ttmain}\label{thm:mainideal} Let $X$ be as in Theorem \ref{thm:maincoordring}. The ideal $I(\tau(X))$ is generated in degree at most $4$, and has minimal generators of degree $4$ if and only if the multiset $\{d_1,\cdots,d_n\}$ contains one of $\{3\}$, $\{2,1\}$ or $\{1,1,1\}$. It has minimal generators of degree $2$ if and only if $\sum_i d_i\geq 4$. $I(\tau(X))$ is generated by quadrics if and only if $X$ is a rational normal curve of degree at least $5$, or $X=SV_{d,1}(\bb{P}^1\times\bb{P}^r)$ with $d\geq 5$. When $X$ is the twisted cubic, it is known classically that $\t(X)$ is a rational quartic surface (see \cites{sri,atiyah,cat-waj} or \cite[Ex.~11.24]{ful-har}). In all other situations when $I(\tau(X))$ is non-zero, it has minimal generators of degree $3$. An explicit description of the minimal generators of $I(\tau(X))$ in all cases is given in Theorems~\ref{thm:mingens3} and~\ref{thm:mingens4}. In the case of a Segre variety the generators are as predicted by Landsberg and Weyman (Conjecture~\ref{conj:lanwey}).
\end{ttmain}

We record for future reference the decomposition into irreducible representations of the module of minimal generators of the ideal of the tangential variety to a Veronese variety (the highest weight vectors are described in Theorems~\ref{thm:mingens3} and~\ref{thm:mingens4}).

\begin{corollary*}
 Let $d\geq 2$, and let $X=SV_d(\bb{P}V^*)$ be a Veronese variety. The module $K_{1,q}(X)$ of minimal generators of $I(\tau(X))$ (first syzygies of $\K[\t(X)]$) of degree $(q+1)$ decomposes as
\[K_{1,q}(X)=\bigoplus_{\ll}(S_{\ll}V)^{\oplus m_{\ll}},\]
where $m_{\ll}\in\{0,1\}$, with $m_{\ll}=1$ precisely in the following cases:
\begin{enumerate}
 \item $q=1$: $\ll=(2d-k,k)$ for $4\leq k\leq d$, $k$ even.
 \item $q=2$: $\ll=(3d-4,2,2)$ for $d\geq 2$, $\ll=(4,4,1)$ when $d=3$, and $\ll=(6,6)$ when $d=4$.
 \item $q=3$: $\ll=(6,6)$ when $d=3$.
\end{enumerate}
\end{corollary*}

For a projective variety $X$, the relationship between $\t(X)$ and $\s_2(X)$ can be described as follows, as a consequence of Fulton and Hansen's Connectedness Theorem and related to Zak's Theorem on Tangencies: the two varieties are either equal (the degenerate situation), or they both have the expected dimension and the tangential variety is a hypersurface in the secant variety (the typical situation). In this sense, all non-trivial Segre--Veronese varieties are typical except those associated to matrices or symmetric matrices \cite[Theorem~4.2]{abo12}. We can easily recover this fact as a consequence of our results: Theorem~\ref{thm:mainideal} shows that $I(\tau(X))$ has minimal generators of degree $2$ when $\sum_i d_i\geq 4$, and minimal generators of degree $4$ when $\sum_i d_i=3$; since $I(\s_2(X))$ only has minimal generators of degree $3$ \cite{raiGSS}, it follows that we are in the typical situation as long as $\sum_i d_i\geq 3$. The remaining cases $n=2,d_1=d_2=1$, resp. $n=1,d_1=2$, correspond to the spaces of matrices, resp. symmetric matrices, which are degenerate. Alternatively, one can compare $\tau(X)$ with $\s_2(X)$ by comparing the descriptions of their homogeneous coordinate rings from Theorem~\ref{thm:maincoordring} and \cite[Theorem~4.1]{raiGSS}.

The methods we employ in this paper are to apply the techniques of the second author, which were used to determine the ideal of the secant varieties to the Segre--Veronese varieties~\cite{raiGSS}. This represents a departure from how this ideal was studied in the past in that we study the ideal via combinatorial properties of representations of the symmetric group.  Since these techniques were successful in studying the ideal of the secant varieties to Segre--Veronese varieties, it is natural to expect that they would also apply to the tangential varieties, and this is what we show in the remainder of the paper. The main new ingredient is Theorem~\ref{thm:idealdtab}, which is the main application of the results in \cite{rai-yng}. This tool wasn't necessary in \cite{raiGSS} because of the simplicity of the equations coming from minors of flattenings, but we expect it to be relevant in other situations where the structure of the generating set of the ideal is more involved. In this paper we develop a more functorial setting, which allows us to make the ring operations more apparent in the generic setting (this was only implicit in~\cite{raiGSS}).

For any variety, the ideal of defining equations allows one to test membership on that variety. This problem for various types of varieties is addressed in depth in \cite{jmltensorbook}, and the tangential variety is discussed in Chapter~8 in particular. In \cite{oedingJPAA}, the first author pointed out applications of the tangential variety of an $n$--factor Segre where the equations allow one to answer the question of membership for the following sets: the set of tensors with border rank $2$ and rank $k\leq n$ (the secant variety is stratified by such tensors \cite{Bernardi}), a special Context--Specific Independence model, and a certain type of inverse eigenvalue problem.  Another recent instance of the tangential variety of a Segre variety is in \cite[\S~4]{Sturmfels-Zwiernik}, where the authors showed that after a non--linear change of coordinates to cumulant coordinates, the tangential variety becomes a toric variety, and they computed its ideal in the case of $5$ factors in cumulant coordinates.

We summarize the structure of the paper as follows. In Section~\ref{sec:eqnstau} we give an explicit description of the tangential variety $\tau(X)$ of a Segre--Veronese variety $X$, and characterize its equations in terms of linear algebra. In Section~\ref{sec:generic} we set up the generic case, introducing the generic coordinate ring of $X$. In Section \ref{subsec:eqnstangential} we describe a method to compute the generic equations and the generic coordinate ring of $\tau(X)$. In Section~\ref{subsec:tabloids} we define tabloids (which were called $n$--tableaux in \cite{raiGSS}), and in Section~\ref{subsec:graphs} we recall the language of graphs and its relation to tabloids from \cite{raiGSS}. In Section~\ref{subsec:MCB} we collect a series of results on the generic equations of $\s_2(X)$ that will be relevant to the study of $\tau(X)$. In Section~\ref{sec:graph-lanwey} we give a simple graphical description of the generic equations of $\tau(X)$ and deduce Theorem~\ref{thm:maincoordring}. Theorem~\ref{thm:mainideal} is a consequence of the very explicit results in Section~\ref{sec:mingensI}. In Section~\ref{subsec:Ileq4} we show that the ideal of the tangential variety is generated in degree at most $4$, and in Sections~\ref{subsec:mingens3} and~\ref{subsec:mingens4} we determine the tabloids that minimally generate the generic version of this ideal.

\subsection*{Notation} $\K$ will always denote a field of characteristic zero. We denote the set $\{1,2,\cdots,r\}$ by $[r]$. If $\mu=(\mu_1\geq\mu_2\geq\cdots)$ is a partition of $r$ (written $\mu\vdash r$) and $W$ a vector space, then $S_{\mu}W$ (resp. $[\mu]$) denotes the irreducible representation of the general linear group $GL(W)$ (resp. of the symmetric group $\SS_{r}$) corresponding to $\mu$. If $\mu=(r)$, then $S_{\mu}W$ is $\rm{Sym}^r(W)$ and $[\mu]$ is the trivial $\SS_{r}$--representation. If $\mu=(1^r)$, $S_{\mu}=\bigwedge^r$ is the exterior power functor. The $GL(W)$-- (resp. $\SS_r$--) representations $U$ that we consider decompose as $U = \bigoplus_{\mu} U_{\mu}$ where $U_{\mu}\simeq (S_{\mu}W)^{m_{\mu}}$ (resp. $U_{\mu}\simeq[\mu]^{m_{\mu}}$) is the \defi{$\mu$--isotypic component} of $U$. We make the analogous definitions when we work over products of general linear (resp. symmetric) groups, replacing partitions by $n$--tuples of partitions (called \defi{$n$--partitions} and denoted by $\vdash^n$). We will often work with a collection $V=(V_1,\cdots,V_n)$ of vector spaces. If $\N=(N_1,\cdots,N_n)$ and $\ll\vdash^n\N$, we write $S_{\ll}V$ for $S_{\ll^1}V_1\oo\cdots\oo S_{\ll^n}V_n$, and $\Sym^{\N}V$ for $\Sym^{N_1}V_1\oo\cdots\oo\Sym^{N_n}V_n$. For an introduction to the representation theory of general linear and symmetric groups, see~\cite{ful-har}. For a short account of the relevant facts needed in this paper, see~\cite[Section~2C]{raiGSS}. We will sometimes write $\ul{a}\vdash r$ for a partition of $r$, to mean an $n$--tuple $\ul{a}=(a_1,\cdots,a_n)$ with $a_1+\cdots+a_n=r$. We write $\1=(1,\cdots,1)$. For a graph $G$ we will write $V(G)$ for its set of vertices, and $E(G)$ for its set of edges. Even though our graphs will always be oriented, when we talk about paths or cycles we do not take into account orientation. If $P$ is a path or a cycle in a graph, we define the \defi{length} $l(P)$ of $P$ to be the number of edges in $P$.

\section{Equations of the tangential variety of a Segre--Veronese variety}\label{sec:eqnstau}

For the rest of the paper we fix a collection $\d=(d_1,\cdots,d_n)$ of positive integers. We consider an $n$--tuple $V=(V_1,\cdots,V_n)$ of vector spaces over a field of characteristic zero, and let $X$ denote the \defi{Segre--Veronese} variety obtained as the image of the embedding
\[SV_{\d}:\bb{P}V_1^*\times\cdots\times\bb{P}V_n^*\to\bb{P}(\Sym^{\d}V^*)\]
given by
\[([e_1],\cdots,[e_n])\mapsto[e_1^{d_1}\o\cdots\o e_n^{d_n}].\]
The cone $\widehat{\t(X)}$ over the tangential variety $\t(X)$ is the set of tensors obtained as
\[\lim_{t\to 0}\frac{(e_1+tf_1)^{d_1}\o\cdots\o(e_n+tf_n)^{d_n} - e_1^{d_1}\o\cdots\o e_n^{d_n}}{t}\]
\[=d_1\cdot (e_1^{d_1-1}f_1\o e_2^{d_2}\o \cdots\o e_n^{d_n}) + d_2\cdot (e_1^{d_1}\o e_2^{d_2-1}f_2\o \cdots\o e_n^{d_n}) + \cdots + d_n\cdot (e_1^{d_1}\o\cdots\o e_{n-1}^{d_{n-1}}\o e_n^{d_n-1}f_n),\]
where $e_i,f_i\in V_i$. Replacing $f_i$ by $f_i/d_i$, we see that $\widehat{\t(X)}$ is the image of the map
\[s=s(V):(V_1^*\times\cdots\times V_n^*)\times(V_1^*\times\cdots\times V_n^*)\lra \Sym^{\d} V^*,\]
\[(e_1,\cdots,e_n,f_1,\cdots,f_n)\lra\sum_{i=1}^n e_1^{d_1}\o\cdots\o e_i^{d_i-1}f_i\o\cdots\o e_n^{d_n}.\]
Here we write $s(V)$ when we want to emphasize the dependence on the collection $V$ of vector spaces. $s$ corresponds to a ring map
\[\pi=\pi(V):\Sym(\Sym^{\d}V)\lra (\Sym(V_1)\oo\cdots\oo\Sym(V_n))^{\o 2},\]
which acts on generators as follows. We write $\phi_i:\Sym^{d_i}V_i\to\Sym^{d_i-1}V_i\oo V_i$ for the dual of the natural multiplication map $\Sym^{d_i-1}V_i^*\oo V_i^*\to\Sym^{d_i}V_i^*$. For a collection of monomials $m_i\in\Sym^{d_i}V_i$, we write $\phi_i(m_i)=\sum_j u_{i,j}\oo v_{i,j}$. With this notation, we have
\[\pi(m_1\o\cdots\o m_n) = \sum_{i=1}^n \sum_j (m_1\o\cdots\o m_{i-1}\o u_{i,j}\o m_{i+1}\o \cdots\o m_n) \o (1\o\cdots\o 1\o v_{i,j}\o 1 \o \cdots\o 1).\]
In degree $r$, $\pi$ restricts to a map
\[\pi_r=\pi_r(V):\Sym^r(\Sym^{\d}V)\lra\bigoplus_{\ul{a}\vdash r}(\Sym^{r\d-\ul{a}}V)\o(\Sym^{\ul{a}}V).\]
We decompose $\pi_r$ as $\pi_r(V) = \bigoplus_{\ul{a}\vdash r}\pi_{\ul{a}}(V)$ where
\[\pi_{\ul{a}}(V):\Sym^r(\Sym^{\d}V)\lra(\Sym^{r\d-\ul{a}}V)\o(\Sym^{\ul{a}}V).\]
We will give a very explicit description of $\pi_{\ul{a}}(V)$ in (\ref{eqn:pia}) below. We write $\pi_{\ul{a}}(V)$ to distinguish this map from its generic version $\pi_{\ul{a}}$ which is introduced in Section \ref{subsec:eqnstangential}.

Let $m_j = \rm{dim}(V_j)$ and let $\mc{B}_j=\{x_{i,j}:i\in [m_j]\}$ be a basis for $V_j$. For an $n$--tuple $\ul{a}$ of positive integers, the vector space $\Sym^{\ul{a}}V$ has a basis $\mc{B}=\mc{B}_{\ul{a}}$ consisting of tensor products of monomials in the elements of the bases $\mc{B}_1,\cdots,\mc{B}_n$. We write this basis, suggestively, as
\[\mc{B} = \Sym^{a_1}\mc{B}_1\o\cdots\o\Sym^{a_n}\mc{B}_n.\]
We can index the elements of $\mc{B}$ by $n$--tuples $\a=(\a_1,\cdots,\a_n)$ of multisets $\a_i$ of size $a_i$ with entries in $\{1,\cdots,m_i=\rm{dim}(V_i)\}$, as follows. The $\a$--th element of the basis $\mc{B}$ is
\[z_{\a} = (\prod_{i_1\in\a_1}x_{i_1,1})\o\cdots\o(\prod_{i_n\in\a_n}x_{i_n,n}).\]

When $\ul{a}=\d$, we think of $z_{\a}$ as a linear form in $S=\Sym(\Sym^{\d}V)$, so that $S=\K[z_{\a}]$ is a polynomial ring in the variables $z_{\a}$. We identify each $z_{\a}$ with an $1\times n$ block with entries $\a_1,\cdots,\a_n$:
\[
 z_{\a} = \begin{array}{|c|c|c|c|}
 \hline
 \a_1 & \a_2 & \cdots & \a_n \\ \hline
 \end{array}.
\]
We represent a monomial $m = z_{\a^1}\cdots z_{\a^r}$ of degree $r$ as an $r\times n$ block $M$, whose rows correspond to the variables $z_{\a^i}$ in the way described above.
\[
 m \equiv M = \begin{array}{|c|c|c|c|}
 \hline
 \a^1_1 & \a^1_2 & \cdots & \a^1_n \\ \hline
 \a^2_1 & \a^2_2 & \cdots & \a^2_n \\ \hline
 \vdots & \vdots & \ddots & \vdots \\ \hline
 \a^r_1 & \a^r_2 & \cdots & \a^r_n \\ \hline
 \end{array}
\]
Note that the order of the rows is irrelevant, since the $z_{\a^i}$'s commute.

For $\ul{a}\vdash r$, we represent a monomial $z_{\b}\o z_{\c}$ in the target of the map $\pi_{\ul{a}}$ as a $2\times n$ block
\[
\begin{array}{|c|c|c|c|}
 \hline
 \b_1 & \b_2 & \cdots & \b_n \\ \hline
 \c_1 & \c_2 & \cdots & \c_n \\ \hline
 \end{array},
\]
where $\b_i$ are multisets of size $rd_i-a_i$ and $\c_i$ are multisets of size $a_i$ (the order of the rows is now important!). The map $\pi_{\ul{a}}$ can then be written in terms of blocks as
\begin{equation}\label{eqn:pia}
M\lra\sum_{\substack{A_1\sqcup\cdots\sqcup A_n = [r]\\|A_j| = a_j}}
 \begin{array}{|c|c|c|}
  \hline
 \cdots & \bigcup_{i\notin A_j}\a_j^i & \cdots\\ \hline
 \cdots & \bigcup_{i\in A_j}\a_j^i & \cdots\\ \hline
 \end{array}.
\end{equation}

\begin{example}\label{ex:piablock} Assume that $r=3$, $n=4$, $\d=\1$ (i.e. $X$ is a Segre variety), and $V_i$ are vector spaces of dimensions at least three. Let $\ul{a}\vdash r$ with $a_1=2$, $a_2=a_3=0$, $a_4=1$, and consider the monomial
\[m=z_{(1,1,2,3)}\cdot z_{(1,2,1,2)}\cdot z_{(2,3,2,1)}\in \Sym^3(V_1\o V_2\o V_3\o V_4).\]
We can represent it as the $3\times 4$ block
\[M = \begin{array}{|c|c|c|c|}
 \hline
1 & 1 & 2 & 3 \\ \hline
1 & 2 & 1 & 2 \\ \hline
2 & 3 & 2 & 1 \\ \hline
\end{array}.\]
The map $\pi_{\ul{a}}(V)$ sends
\[M\lra \begin{array}{|c|c|c|c|}
 \hline
2 & 1,2,3 & 1,2,2 & 2,3 \\ \hline
1,1 &  &  & 1 \\ \hline
\end{array}+\begin{array}{|c|c|c|c|}
 \hline
1 & 1,2,3 & 1,2,2 & 1,3 \\ \hline
1,2 &  &  & 2 \\ \hline
\end{array}+\begin{array}{|c|c|c|c|}
 \hline
1 & 1,2,3 & 1,2,2 & 1,2 \\ \hline
1,2 &  &  & 3 \\ \hline
\end{array}.
\]

\begin{example}\label{ex:piaVero}
 Assume now that $r=1$, $n=1$, $\d=(4)$ (i.e. $X$ is a Veronese variety), and $V_1$ is a vector space of dimension at least three. There is only one partition $\ul{a}\vdash 1$, so that $s_1^{\#}(V)=\pi_{\ul{a}}(V)$. Consider the linear form $z_{\{1,1,2,3\}}\in\Sym^1(\Sym^{4}(V_1))$, corresponding to the $1\times 1$ block
\[M=\begin{array}{|c|}
 \hline
1 , 1 , 2 , 3 \\ \hline
\end{array}.\]
The map $\pi_{\ul{a}}(V)$ sends
\[M\lra\begin{array}{|c|}
 \hline
1 , 1 , 2 \\ \hline
 3 \\ \hline
\end{array}+\begin{array}{|c|}
 \hline
1 , 1 , 3 \\ \hline
 2 \\ \hline
\end{array}+\begin{array}{|c|}
 \hline
1 , 2 , 3 \\ \hline
 1 \\ \hline
\end{array}.
\]
\end{example}

The above discussion implies the following

\begin{proposition}\label{prop:eqnstau}
 The equations of degree $r$ of the tangential variety $\tau(X)$ of the Segre--Veronese variety $X$ are given by the subspace $I_r(V)\subset\Sym^{r}(\Sym^{\d}V)$ defined as
\[I_r(V)=\bigcap_{\ul{a}\vdash r}\rm{Ker}(\pi_{\ul{a}}(V)).\]
\end{proposition}
\end{example}

\section{The generic case}\label{sec:generic}

The material in this section is based on~\cite[Section~3B]{raiGSS}. We introduce a more functorial terminology inspired by \cites{snowden,sam-snowden,sam-snowden-tca} (see also \cite{rai-yng}). Let $\d=(d_1,\cdots,d_n)$ be a sequence of positive integers as before. For a nonnegative integer $r$ we write $\r$ for the $n$--tuple $(r_1,\cdots,r_n)$, where $r_j=rd_j$. We write $\SS_{\r}=\SS_{r_1}\times\cdots\times\SS_{r_n}$. We dentote by $Set^{\d}_r$ the category whose objects are $n$--tuples of sets $\A=(A_1,\cdots,A_n)$ with $|A_j|=r_j$, and morphisms are componentwise bijections. We write $\SS_{\A}$ for the product $\SS_{A_1}\times\cdots\times \SS_{A_n}$, where $\SS_{A_i}$ is the group of permutations of $A_i$ (alternatively, $\SS_{\A}=\Hom_{Set^{\d}_r}(\A,\A)$). We let $Set^{\d}=\bigcup_{r\geq 0}Set^{\d}_r$, and let $Vec$ denote the category of finite dimensional $\K$--vector spaces. Given any functor $\FF:Set^{\d}_r\to Vec$ and any $\A\in Set^{\d}_r$, $\FF(\A)=\FF_{\A}$ is a $\SS_{\A}$--representation. If 
$\ll\vdash^n\r$, 
we denote by $m_{\ll}(\FF)$ the \defi{multiplicity} of the irreducible $\SS_{\A}$--representation $[\ll]$ inside $\FF_{\A}$. We write $\FF_{\ll}$ for the subfunctor that assigns to $\A$ the $\ll$--isotypic component $(\FF_{\A})_{\ll}$ of $\FF_{\A}$ (note that $(\FF_{\ll})_{\A}=(\FF_{\A})_{\ll}$).

\begin{definition}[The generic coordinate ring of $\bb{P}(\Sym^{\d}V$)]\label{def:urd} 
Let $\SSS^{\d}:Set^{\d}\to Vec$ be the functor which assigns to an element $\A\in Set^{\d}_r$ the vector space $\SSS^{\d}_{\A}$ with basis consisting of monomials
\[z_{\a} = z_{\a^1}\cdots z_{\a^r},\]
in commuting variables $z_{\a^j}$, where $\a^j=(\a^j_1,\cdots,\a^j_n)$ are $n$--tuples of sets with $|\a^j_i|=d_i$, and for each $i$ we have $\a_i^1\cup\cdots\cup\a_i^r=A_i$. Alternatively, $\SSS^{\d}_{\A}$ has a basis consisting of $r\times n$ blocks $M$, where each column of $M$ yields a partition of the set $A_j$ into $r$ equal parts. Note that when $r=1$, $\SSS^{\d}_{\A}$ is $1$--dimensional and has a distinguished generator $z_{\A}$. When $A_j=[r_j]$ for all $j$, we write $\SSS^{\d}_{\r}$ instead of $\SSS^{\d}_{\A}$. As before, we identify two blocks if they differ by permutations of their rows. $\SSS^{\d}_{\r}$ is the \defi{generic version} of the representation $\Sym^{r}(\Sym^{\d}V)$ in the sense of \cite[Section~3B]{raiGSS}. We will write $\SSS$ instead of $\SSS^{\d}$. 
\begin{example}\label{ex:urd}
 For $n=2$, $\d=(2,1)$, $r=4$, $A_1=[8]$, $A_2=[4]$, a typical element of $\SSS_{\A}$ is
\[
\begin{array}{|c|c|}
 \hline
1 , 6 & 1 \\ \hline
2 , 8 & 4 \\ \hline
4 , 7 & 2 \\ \hline
3 , 5 & 3 \\ \hline
\end{array} = 
z_{\{1,6\},1}\cdot z_{\{2,8\},4}\cdot z_{\{4,7\},2}\cdot z_{\{3,5\},3} = 
z_{\{3,5\},3}\cdot z_{\{1,6\},1}\cdot z_{\{4,7\},2}\cdot z_{\{2,8\},4} =
\begin{array}{|c|c|}
 \hline
3 , 5 & 3 \\ \hline
1 , 6 & 1 \\ \hline
4 , 7 & 2 \\ \hline
2 , 8 & 4 \\ \hline
\end{array}
\]
\end{example}

There is a natural multiplication map $\mu_{\A,\B}:\SSS_{\A}\oo \SSS_{\B}\to \SSS_{\A\sqcup\B}$, where $\A\sqcup\B$ denotes the componentwise disjoint union $(A_1\sqcup B_1,\cdots,A_n\sqcup B_n)$. We will often write $\mu_{A,B}(x\oo y)$ simply as $x\cdot y$. We call $\SSS$ a \defi{generic algebra} (or simply an \defi{algebra} when no confusion is possible). An \defi{element of} $\SSS$ is an element $w\in \SSS_{\A}$ for some $\A\in Set^{\d}$. An \defi{ideal} $\II\subset\SSS$ is simply a subfunctor which is closed under multiplication by elements of $\SSS$, i.e. it has the property that $\II_{\A}\cdot \SSS_{\B}\subset \II_{\A\sqcup\B}$ for all $\A,\B\in Set^{\d}$. When $\mc{W}$ is a collection of elements of $\SSS$ we define the \defi{ideal $\II(\mc{W})$ generated by $\mc{W}$} to be the smallest ideal $\II$ that contains the elements of $\mc{W}$. The \defi{maximal ideal} $\m\subset \SSS$ is given by $\m_{\A}=\SSS_{\A}$ whenever $\A\in Set^{\d}_r$, $r>0$, and $\m_{\ul{\emptyset}}=0$ (where $\ul{\emptyset}$ is 
the unique element of $Set^{\d}_0$). Note that $\m$ is generated by $z_{\A}\in \SSS_{\A}$ for any $\A\in Set^{\d}_1$.
\end{definition}

\begin{example}[The generic ideal of the subspace variety, see also {\cite[Theorem~3.1]{lan-wey-secant}}]\label{ex:subspace}
 The ideal $\II^{<3}\subset\SSS$ generated by the isotypic components $\SSS_{\ll}$, where $\ll$ runs over the $n$--partitions with some $\ll^j$ having at least three parts is the generic version of the ideal of the \defi{subspace variety} obtained by taking the union of $SV_{\d}(\bb{P}W_1^*\times\cdots\times\bb{P}W_n^*)$ where the $W_i$'s run over all $2$--dimensional quotients of the $V_i$'s. If $\ll$ is an $n$--partition such that each $\ll^j$ has at most two parts, then $\II^{<3}_{\ll}=0$, otherwise $\II^{<3}_{\ll}=\SSS_{\ll}$.
\end{example}

\subsection{The generic equations of the tangential variety}\label{subsec:eqnstangential}

In this section we introduce a new (generic) algebra $\SSS'$ together with an algebra map $\pi:\SSS\to \SSS'$ which is the generic version of the map $\pi(V)$ defined in Section \ref{sec:eqnstau}.

\begin{definition}\label{def:ua}
 Let $\SSS':Set^{\d}\to Vec$ be the functor which assigns to an element $\A\in Set^{\d}_r$ the vector space $\SSS'_{\A}=\bigoplus_{\ul{a}\vdash r}\SSS'_{\A,\ul{a}}$ where $\SSS'_{\A,\ul{a}}$ is the vector space with basis consisting of expressions $z_{\b,\c}=z_{\b}\o z_{\c}$, where $\b=(\b_1,\cdots,\b_n)$, $\c=(\c_1,\cdots,\c_n)$, with $\b_j\cup\c_j = A_j$, $|\b_j|=rd_j-a_j$, $|\c_j|=a_j$. We can represent $m$ as a $2\times n$ block
\[M = 
\begin{array}{|c|c|c|c|}
 \hline
 \b_1 & \b_2 & \cdots & \b_n \\ \hline\hline
 \c_1 & \c_2 & \cdots & \c_n \\ \hline
 \end{array}.
\]
$\SSS'$ is itself a ring, where the multiplication is defined on blocks (and extended linearly) by
\[\begin{array}{|c|c|c|c|}
 \hline
 \b_1 & \b_2 & \cdots & \b_n \\ \hline\hline
 \c_1 & \c_2 & \cdots & \c_n \\ \hline
 \end{array}\cdot
\begin{array}{|c|c|c|c|}
 \hline
 \b'_1 & \b'_2 & \cdots & \b'_n \\ \hline\hline
 \c'_1 & \c'_2 & \cdots & \c'_n \\ \hline
 \end{array} =
\begin{array}{|c|c|c|c|}
 \hline
 \b_1 \cup \b_1' & \b_2 \cup \b_2' & \cdots & \b_n \cup \b_n' \\ \hline\hline
 \c_1 \cup \c_1' & \c_2 \cup \c_2' & \cdots & \c_n \cup \c_n' \\ \hline
 \end{array}.
\]
When $A_j=[r_j]$ for all $j$, we simply write $\SSS'_{\ul{a}}$ instead of $\SSS'_{\A,\ul{a}}$. $\SSS'_{\ul{a}}$ is the generic version of $\Sym^{r\d-\ul{a}}V\o\Sym^{\ul{a}}V$.
\end{definition}

\begin{definition}[Generic $\pi$]\label{def:genericss}
 For $\A\in Set^{\d}_r$ and $\ul{a}\vdash r$ we consider the map 
\[\pi_{\A,\ul{a}}:\SSS_{\A}\lra \SSS'_{\A,\ul{a}}\subset \SSS'_{\A}\]
defined on blocks according to the formula (\ref{eqn:pia}). The association $\A\mapsto\pi_{\A,\ul{a}}$ defines a natural transformation $\pi_{\ul{a}}$ between the functors $\SSS|_{Set^{\d}_r}$ and $\SSS'|_{Set^{\d}_r}$. Letting $\pi_r=\bigoplus_{\ul{a}\vdash r}\pi_{\ul{a}}$ and putting all the $\pi_r$'s together, we get a natural transformation $\pi=\bigoplus_{r\geq 0}\pi_r:\SSS\to\SSS'$ which is in fact an algebra map. We call the image of $\pi$ (which is a subalgebra of $\SSS'$) the \defi{generic coordinate ring of the tangential variety $\tau(X)$}.
\end{definition}

\begin{example}\label{ex:piageneric} Assume that $r=n=3$, $\d=(1,1,2)$, $A_1=A_2=[3]$, $A_3=[6]$, and let $\ul{a}=(2,0,1)\vdash r$. Consider the monomial
\[m=z_{1,1,\{3,6\}}\cdot z_{3,2,\{1,5\}}\cdot z_{2,3,\{2,4\}}\in\SSS_{\A}.\]
We can represent it as the $3\times 3$ block
\[M = \begin{array}{|c|c|c|}
 \hline
1 & 1 & 3 , 6 \\ \hline
3 & 2 & 1 , 5 \\ \hline
2 & 3 & 2 , 4 \\ \hline
\end{array}.\]
The map $\pi_{\ul{a}}$ sends $M$ to
\[
\begin{split}
&\begin{array}{|c|c|c|}
 \hline
2 & 1,2,3 &  1,2,3,5,6 \\ \hline
1,3 &  &  4 \\ \hline
\end{array}\ +\ \begin{array}{|c|c|c|}
 \hline
2 & 1,2,3 &  1,3,4,5,6 \\ \hline
1,3 &  &  2 \\ \hline
\end{array}\ +\ \begin{array}{|c|c|c|}
 \hline
3 & 1,2,3 &  1,2,3,4,6 \\ \hline
1,2 &  &  5 \\ \hline
\end{array}\ + \\
& \begin{array}{|c|c|c|}
 \hline
3 & 1,2,3 &  2,3,4,5,6 \\ \hline
1,2 &  &  1 \\ \hline
\end{array}\ +\ \begin{array}{|c|c|c|}
 \hline
1 & 1,2,3 & 1,2,3,4,5 \\ \hline
2,3 &  &  6 \\ \hline
\end{array}\ +\ \begin{array}{|c|c|c|}
 \hline
1 & 1,2,3 & 1,2,4,5,6 \\ \hline
2,3 &  &  3 \\ \hline
\end{array}.
\end{split}
\]
\end{example}

\begin{definition}[Generic equations of the tangential variety]\label{def:lanweyeqns}
 Let $\II=\rm{Ker}(\pi)$. For a positive integer $r$ we have
\[\II_r=\rm{Ker}(\pi_r)=\bigcap_{\ul{a}\vdash r}\rm{Ker}(\pi_{\ul{a}}).\]
$\II_r$ is the \defi{generic version} of $I_r(V)$ (see Prop.~\ref{prop:eqnstau}) in the sense of \cite[Prop.~3.27]{raiGSS}.
\end{definition}

\begin{example}[See also {\cite[Section~3C]{raiGSS}}]\label{ex:multiplicities}
 It follows from Schur--Weyl duality that the multiplicities $m_{\ll}(\SSS),m_{\ll}(\pi(\SSS))$ and $m_{\ll}(\II_r)$ coincide with the multiplicities of $S_{\ll}V$ inside $\Sym^r(\Sym^{\d}V),\K[\tau(X)]_r$ and $I_r(V)$ respectively (when the dimensions of the vector spaces $V_i$ are sufficiently large).
\end{example}


\subsection{Young tabloids \cite[Section~5]{rai-yng}}\label{subsec:tabloids} 

We caution the reader of a shift in terminology from \cite{raiGSS}: the Young tabloids in this article were called Young $n$--tableaux in \cite{raiGSS}. 

For a partition $\mu\vdash N$, $D_{\mu}=\{(x,y):1\leq y\leq \mu_x\}$ denotes the associated \defi{Young diagram}. A \defi{Young tableau $T$} of shape $\mu$ and entry set $A$ is a bijection $T:D_{\mu}\to A$. We represent Young tableaux (or more generally, an arbitrary function $T:D_{\ll}\to A$) pictorially as a collection of left-justified rows of boxes filled with entries of $A$, with $\ll_i$ boxes in the $i$-th row, as illustrated in the following example: for $A=\{a,b,c,d,e,f,g\}$ and $\mu=(4,2,1)$, we take
\begin{equation}\label{eq:tableau}
\Yvcentermath1 T=\young(cabg,ed,f). 
\end{equation}
We use matrix--style coordinates, so the example above has $T(1,3)=b$. If $\ll\vdash^n\ul{N}$ then $D_{\ll}=(D_{\ll^1},\cdots,D_{\ll^n})$. A \defi{Young $n$--tableau $T$} of shape $\ll$ and entry set $\A=(A_1,\cdots,A_n)$, denoted $T:D_{\ll}\to\A$, is an $n$--tuple $(T^1,\cdots,T^n)$, where each $T^i$ is a Young tableau of shape $\ll^i$ and entry set $A_i$. To any Young $n$--tableau $T$ we associate the \defi{Young symmetrizer} $\mfc_{\ll}(T)$ in the group algebra $\K[\SS_{\A}]$. The \defi{canonical Young $n$--tableau} $T_{\ll}$ of shape $\ll$ is defined by $T_{\ll}^i(x,y)=\ll^i_1+\cdots+\ll^i_{x-1}+y$.

\begin{definition}[Covariants]\label{def:covariants}
Consider a functor $\FF:Set^{\d}\to Vec$, a positive integer $r$, and an $n$--partition $\ll\vdash^n\r$. A choice of $\A\in Set^{\d}_r$ and of a Young $n$--tableau $T:D_{\ll}\to\A$ gives rise to a vector space $\mfc_{\ll}(T)\cdot (\FF_{\A})_{\ll}$ of dimension $m_{\ll}(\FF)$. We call this space a \defi{$\ll$--highest weight space} of $\FF$ and denote it by $\rm{hwt}_{\ll}(\FF)$. We call the elements of $\rm{hwt}_{\ll}(\FF)$ \defi{$\ll$--covariants} of $\FF$. Note that there are choices in the construction of the $\ll$--covariants of $\FF$, but the subfunctor of $\FF$ that they generate is $\FF_{\ll}$, which is independent of these choices.
\end{definition}

\begin{definition}[Young$^{\d}_r$ tabloids]\label{def:yngdtabs}
Consider a positive integer $r$ and an $n$--partition $\ll\vdash^n\r$. A \defi{Young$^{\d}_r$ tableau} $F:D_{\ll}\to[r]$ of shape $\ll$ is an $n$--tuple $(F^1,\cdots,F^n)$, where $F^i:D_{\ll^i}\to [r]$ is a function with the property that $|(F^i)^{-1}(j)|=d_i$ for all $j\in[r]$. If $T:D_{\ll}\to\A$ is a Young $n$--tableau then we write $T\circ F^{-1}$ for the collection $\a=(\a^i_j)$, where $\a^i_j=T^i((F^i)^{-1}(j))$. We define the \defi{Young$^{\d}_r$ tabloid} associated to $F$ to be the collection $[F]$ of covariants $t_F(\A,T)=\mfc_{\ll}(T)\cdot z_{T\circ F^{-1}}\in\SSS_{\A}$, for all choices of $\A\in Set^{\d}_r$ and all Young $n$--tableaux $T:D_{\ll}\to\A$. Note that replacing $F$ by $\s\circ F$ for $\s\in\SS_r$ permutes the $n$--tuples $\a^j=(\a^j_1,\cdots,\a^j_n)$, thus it preserves $z_{\a}$ because of the commutativity of the $z_{\a^j}$'s. It follows that $[F]=[\s\circ F]$ for all $\s\in\SS_r$. Note also that $\mfc_{\ll}(T)=\mfc_{\ll}(\tl{T})$ if for each $j$, $T^j$ and $\tl{T}^j$ differ by 
permutations of columns of the same size, thus $[F]=[\tl{F}]$ when $F$ and $\tl{F}$ differ by such permutations.
\end{definition}

\begin{definition}[Young$^{\d,\ul{a}}_r$ tabloids]\label{def:yngdatabs}
Similarly, we define a \defi{Young$^{\d,\ul{a}}_r$ tableau} $F':D_{\ll}\to\{1,2\}$ of shape $\ll$ to be an $n$--tuple $(F'^1,\cdots,F'^n)$, where $F'^i:D_{\ll^i}\to\{1,2\}$ is a function with the property that $|(F'^i)^{-1}(1)|=rd_i-a_i$ and $|(F'^i)^{-1}(2)|=a_i$ for all $i\in[n]$. We write $T\circ F'^{-1}$ for the pair of $n$--tuples $\b=(\b_1,\cdots,\b_n)$, $\c=(\c_1,\cdots,\c_n)$, defined by $\b_i=T^i((F'^i)^{-1}(1))$, $\c_i=T^i((F'^i)^{-1}(2))$, so $z_{T\circ F'^{-1}}=z_{\b,\c}$. We define as before the \defi{Young$^{\d,\ul{a}}_r$ tabloid} associated to $F'$ to be the collection $[F']$ of covariants $t'_{F'}(\A,T)=\mfc_{\ll}(T)\cdot z_{T\circ F'^{-1}}\in\SSS'_{\A}$, for all choices of $\A\in Set^{\d}_r$ and all Young $n$--tableaux $T:D_{\ll}\to\A$.
\end{definition}

We represent Young$^{\d}_r$ tabloids and Young$^{\d,\ul{a}}_r$ tabloids as $n$--tuples of tableaux as in (\ref{eq:tableau}) with the horizontal lines removed, separated by $\oo$. The use of the symbol $\oo$ is motivated by the natural identification of $\K[\SS_{\A}]$ with the tensor product $\K[\SS_{A_1}]\oo\cdots\oo\K[\SS_{A_n}]$. We will often identify a tabloid $[F]$ with its representative $t_F(\A,T)$. A relation between tabloids $[F_i]$ should be interpreted as a relation between the elements $t_{F_i}(\A,T)$ for a fixed choice of $\A$ and $T$. 

\begin{example}\label{ex:tableaux}
 With notation as in Example \ref{ex:piageneric}, we let $\ll\vdash^3(3,3,6)$ be the $3$--partition with $\ll^1=\ll^2=(2,1)$, $\ll^3=(4,2)$. Taking $T=T_{\ll}$, the canonical Young $n$--tableau, we get
\[\ytableausetup{boxsize=1.25em,tabloids,aligntableaux=center}
[F]=\ytableaushort{13, 2}\oo\ytableaushort{12, 3}\oo\ytableaushort{2313, 21}\longleftrightarrow t_F(\A,T) = \mfc_{\ll}(T)\cdot \begin{array}{|c|c|c|c|}
 \hline
1 & 1 & 3 , 6 \\ \hline
3 & 2 & 1 , 5 \\ \hline
2 & 3 & 2 , 4 \\ \hline
\end{array} = \mfc_{\ll}(T)\cdot M.
\]
The map $\pi_{\ul{a}}$ sends $[F]$ to
\[
\ytableausetup{boxsize=1.25em,tabloids,aligntableaux=center}
\begin{split}
 &\ytableaushort{21,2}\oo\ytableaushort{11,1}\oo\ytableaushort{1112,11}\ +\ \ytableaushort{21,2}\oo\ytableaushort{11,1}\oo\ytableaushort{1211,11}\ +\\
 &\ytableaushort{22,1}\oo\ytableaushort{11,1}\oo\ytableaushort{1111,21}\ +\ \ytableaushort{22,1}\oo\ytableaushort{11,1}\oo\ytableaushort{2111,11}\ +\\
 &\ytableaushort{12,2}\oo\ytableaushort{11,1}\oo\ytableaushort{1111,12}\ +\ \ytableaushort{12,2}\oo\ytableaushort{11,1}\oo\ytableaushort{1121,11}\\
\end{split}
\]
where the $6$ terms above match (in the order they are given) the $6$ terms in Example \ref{ex:piageneric}.
\end{example}

We introduce some notation in order to state the main technical result on ideals generated by Young tabloids. If $X=\{(x_i,y_i):i\in[d]\}$, $X'=\{(x'_i,y'_i):i\in[d]\}$, with $y_1\leq y_2\leq\cdots$, $y'_1\leq y'_2\leq\cdots$, we write $X'\preceq X$ if $y'_i\leq y_i$ for all $i\in[d]$, and $X'\prec X$ if $X'\preceq X$ and $y'_i<y_i$ for some $i$. If $\ul{X}=(X_1,\cdots,X_n)$, $\ul{X}'=(X_1',\cdots,X_n')$, we write $\ul{X}'\preceq\ul{X}$ if $X'_j\preceq X_j$ for all $j\in[n]$, and $\ul{X}'\prec\ul{X}$ if furthermore $X'_j\prec X_j$ for some $j\in[n]$. The typical situation where this notation is relevant for us is when $F,F':D_{\ll}\to[r]$ are Young$^{\d}_r$ tableaux, $\ul{X}=F^{-1}(t)$ for some $t\in [r]$ (i.e. $X_j=(F^j)^{-1}(t)$ for all $j\in[n]$) and $\ul{X}'=F'^{-1}(t)$.

\begin{theorem}[{\cite[Theorem~5.4]{rai-yng}}]\label{thm:idealdtab}
 Let $k\leq r$ be positive integers, let $\ll\vdash r\d$, $\mu\vdash k\d$ be $n$--partitions with $\mu\subset\ll$, and let $F:D_{\ll}\to[r]$ be a Young$^{\d}_r$ tableau with $F^{-1}([k])=D_{\mu}$. Denote by $\mc{F}$ the collection of Young$^{\d}_r$ tableaux $F':D_{\ll}\to [r]$ with the properties
\begin{enumerate}
 \item $F'^{-1}(t)\preceq F^{-1}(t)$ for all $t\in [k]$, with $F'^{-1}(t)\prec F^{-1}(t)$ for at least one $t\in [k]$.
 \item $F'^{-1}([k])=D_{\delta}$ for some $n$--partition $\delta\subset\ll$.
\end{enumerate}
Writing $F_0=F|_{D_{\mu}}$ we have
 \begin{equation}\label{eq:IF0Fpr}
 [F]\in\II([F_0])+\II([F']:F'\in\mc{F}).
 \end{equation}
If we write $\ol{F}$ for the restriction of $F$ to $F^{-1}([k])$ then
 \begin{equation}\label{eq:IF0Fprbar}
 [F]\in\II([F_0])+\II([\ol{F}']:F'\in\mc{F}).
 \end{equation}
\end{theorem}
Condition (1) can be restated simply by saying that when going from $F$ to $F'$, each entry of $F$ contained in $D_{\mu}$ either remains in the same column, or is moved to the left, the latter situation occurring for at least one such entry.

\subsection{Covariants associated to graphs}\label{subsec:graphs}

Let $r>0$ and $\ll\vdash^n\r$ be an $n$--partition such that each $\ll^j$ has at most two parts. 

\begin{definition}[Admissible graphs and tabloids]\label{def:ideal_graph}
An \defi{admissible graph $G$ of shape $\ll$} is an oriented graph $G$ with $r$ vertices, having $\ll^j_2$ edges of color $j$ for each $j=1,\cdots,n$, and such that to any vertex there are at most $d_j$ incident edges of color $j$. We say that a vertex of $G$ is \defi{$j$--saturated} if it has exactly $d_j$ incident edges of color $j$. Given an admissible graph $G$ of shape $\ll$, we construct a Young$^{\d}_r$ tableau $F:D_{\ll}\to[r]$ as follows. We choose an arbitrary labeling of the vertices of $G$ by the elements of $[r]$, and an enumeration $e^j_1,\cdots,e^j_{\ll^j}$ of the edges of $G$ of color $j$. For every edge $e^j_i=(x,y)$ we let $F^j(1,i)=x$, $F^j(2,i)=y$, so that the columns of $F^j$ correspond to the edges of color $j$. We define $[G]$ to be the tabloid $[F]$ associated to $F$. Note that $F$ is well-defined up to permuting the columns within each $F^j$, and up to a relabeling of the entries (replacing $F$ by $\s\circ F$), in particular $[F]$ is independent of the choices made. Given a tabloid $[
F]$, we can reverse the above construction to obtain an admissible graph $G$ of shape $\ll$ with the property that $[G]=[F]$. We will often identify graphs and Young tabloids when they are related by the construction described above. If $\mc{G}$ is a family of admissible graphs, we write $\II(\mc{G})$ for the ideal $\II([G]:G\in\mc{G})$ generated by the tabloids associated to the graphs in $\mc{G}$.
\end{definition}

\begin{example}\label{ex:graphtab}
 The graph below corresponds to the Young tabloid $[F]$ in Example \ref{ex:tableaux}:
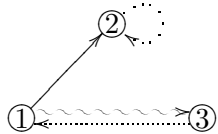
\[
 \xymatrix{
& \ccircle{2} \ar@(ur,dr)@{..>} & \\
\ccircle{1} \ar[ur] \ar@{~>}[rr]<.5ex> & & \ccircle{3} \ar@{.>}[ll]<.5ex>\\
}
\]
where color $1$ corresponds to $\xymatrix{\ar[r]&}$, color $2$ to $\xymatrix{\ar@{~>}[r]&}$, and color $3$ to $\xymatrix{\ar@{..>}[r]&}$.
\end{example}

\begin{example}\label{ex:LW4graph}
Below are examples of generic Landsberg--Weyman equations (see Conjecture \ref{conj:lanwey}) expressed both as Young tabloids (see~\cite[\S~3.2]{oedingJPAA}) and as graphs for $n\leq 4$. Colors $1,2,3$ are as before, and color $4$ corresponds to $
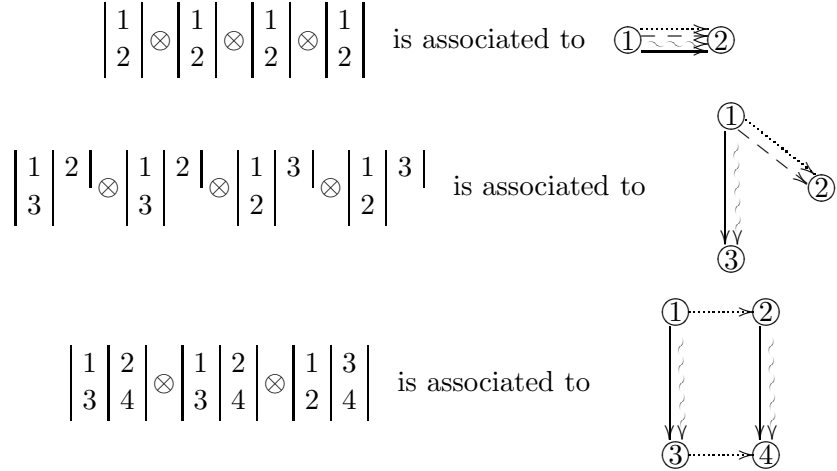
\xymatrix{\ar@{-->}[r]&}$.

\[\ytableausetup{boxsize=1.25em,tabloids,aligntableaux=center}
\ytableaushort{1,2}\oo\ytableaushort{1,2}\oo\ytableaushort{1,2}\oo\ytableaushort{1,2}
\quad\text{is associated to}\quad
\xymatrix@=25pt{
\ccircle{1} \ar[r]<-.9ex> \ar@{~>}[r]<-.3ex> \ar@{-->}[r]<.3ex> \ar@{.>}[r]<.9ex> & \ccircle{2} & \\
}
\]

\[
\ytableausetup{boxsize=1.25em,tabloids,aligntableaux=center}
\xymatrixrowsep{.3pc}
\xymatrix{
& \ccircle{1} \ar[dd]<-.5ex> \ar@{~>}[dd]<.5ex> \ar@{-->}[dr]<-.5ex> \ar@{.>}[dr]<.5ex> & & \\
{\ytableaushort{12,3}\oo\ytableaushort{12,3}\oo\ytableaushort{13,2}\oo\ytableaushort{13,2}\quad\text{is associated to}}& & \ccircle{2} & \\
&\ccircle{3}
}\]

\[
\ytableausetup{boxsize=1.25em,tabloids,aligntableaux=center}
\xymatrixrowsep{.3pc}
\xymatrix{
& \ccircle{1} \ar[dd]<-.5ex> \ar@{~>}[dd]<.5ex> \ar@{.>}[r]  & \ccircle{2} \ar[dd]<-.5ex> \ar@{~>}[dd]<.5ex> & \\
{\ytableaushort{12,34}\oo\ytableaushort{12,34}\oo\ytableaushort{13,24}\quad\text{is associated to}} & &  & \\
&\ccircle{3} \ar@{.>}[r] & \ccircle{4} & \\
}\]
\end{example}

From now on, unless specified otherwise, we assume that all $n$--partitions $\ll$ have the property that $\ll^j=(\ll^j_1\geq\ll^j_2)$ has at most two parts for every $j$. We recall the \defi{straightening laws} (or \defi{shuffling relations}) that Young tabloids satisfy \cite[Chapter~2]{weyman}:

\begin{lemma}[{\cite[Lemma~4.7]{raiGSS}}]\label{lem:fundrels} The following relations between Young tabloids hold (we suppress from the notation the parts of the Young tabloids that don't change, and only illustrate the relevant subtabloids)

(a) $\ytableausetup{boxsize=1.25em,tabloids,aligntableaux=center}\ytableaushort{x,y}=-\ytableaushort{y,x}$, in particular $\Yvcentermath1\ytableaushort{x,x}=0$.

(b) $\ytableausetup{boxsize=1.25em,tabloids,aligntableaux=center}\ytableaushort{xz,y}=\ytableaushort{xy,z}+\ytableaushort{zx,y}$.
\end{lemma}
Interpreted in terms of graphs, part (a) says that reversing an arrow changes the sign of a graph (when viewed as a covariant of $\SSS$). Part (b) can be depicted, in the case when the column of $z$ has size one, as the truncated Pl\"ucker relation
\begin{equation}\label{eq:3plucker}
\begin{aligned}
 \xymatrix@=15pt{
& \ccircle{x} \ar[dr]\\
& & \ccircle{y}  \\
&\ccircle{z}
}
 \xymatrix@=15pt{
& \ccircle{x} \ar[dd]\\
\quad = \quad & & \ccircle{y}  \\
&\ccircle{z}
}
 \xymatrix@=15pt{
& \ccircle{x} \\
\quad + \quad & & \ccircle{y}  \\
&\ccircle{z}\ar[ur]
}
\end{aligned}
\end{equation}
When the column of $z$ has size two, say it is equal to $\ytableaushort{z,t}$, part (b) becomes the Pl\"ucker relation
\begin{equation}\label{eq:plucker}
\begin{aligned}
 \xymatrix@=10pt{
\ccircle{x} \ar[dd] & & \ccircle{z} \ar[dd] & \\
 & & & \quad = \quad \\
\ccircle{y} & & \ccircle{t} & \\
} 
 \xymatrix@=10pt{
\ccircle{x} \ar[rr] & & \ccircle{z} & \\
 & & & \quad + \quad \\
\ccircle{y} \ar[rr] & & \ccircle{t} &\\
} 
 \xymatrix@=10pt{
\ccircle{x} \ar[ddrr] & & \ccircle{z} \ar[ddll]\\
 & & \\
\ccircle{y} & & \ccircle{t}\\
} 
\end{aligned}
\end{equation}

\begin{remark}[see also the inductive step of the proof of {\cite[Prop.~4.17]{raiGSS}}]\label{rem:subgraphs} Given a graph $G$ it will often be convenient to focus on a subgraph $G'$ and the Pl\"ucker relations involving its edges. Some care is needed in order to lift the relations on $G'$ to relations on $G$. The only circumstance where a problem may occur is if we remove the edges in $E(G)\setminus E(G')$, write down a relation of type (\ref{eq:3plucker}) for $G'$, and then put the edges back into the graphs involved in the relation. If the edge $(x,y)$ in (\ref{eq:3plucker}) had color $j$ and $\ccirc{z}$ were a $j$--saturated vertex of $G$, then this procedure would express $[G]$ as a $[G_1]+[G_2]$ with $G_1,G_2$ not admissible. To fix this, note that the only way (\ref{eq:3plucker}) could be a valid relation on $G'$ is if when going from $G$ to $G'$ we removed some edges of color $j$ incident to $\ccirc{z}$. But then we can lift relation (\ref{eq:3plucker}) on $G'$ to a relation of type (\ref{eq:plucker})
 on $G$ that coincides with (\ref{eq:3plucker}) when we forget the edges in $E(G)\setminus E(G')$ (i.e. in (\ref{eq:plucker}) we use an edge $(z,t)\in E(G)\setminus E(G')$ of color $j$).
\end{remark}

The following is a consequence of Theorem \ref{thm:idealdtab}, restated in the language of graphs:
\begin{proposition}[{\cite[Proposition~5.7]{rai-yng}}]\label{prop:idealgraphs}
 If $G$ is an admissible graph of shape $\ll$ and $G'$ is a subgraph containing all the edges of $G$, then $[G]\in\II(G')$. More generally, fix a subset $V\subset V(G)$, $|V|=k$, and let $G'$ be the subgraph with vertex set $V$ containing all the edges of $G$ joining vertices in $V$. Let $\mc{G}$ be the collection of admissible graphs with $k$ vertices that contain $G'$ and have at most $\ll^j$ edges of color $j$ for every $j\in[n]$. We have $[G]\in\II(\mc{G})$.
\end{proposition}

\subsection{MCB--graphs and graphs containing triangles}\label{subsec:MCB} In this section we recall some results from \cite{raiGSS} which can be rephrased (in the language of Definitions \ref{def:urd}, \ref{def:genericss} and \ref{def:ideal_graph}) by saying that the ideal of generic equations of the secant variety $\s_2(X)$ of a Segre--Veronese variety $X$ is generated by graphs containing triangles. Furthermore, in \cite{raiGSS} a description of the multiplicities of the irreducible representations appearing in the coordinate ring of $\s_2(X)$ is given in terms of certain graphs called \defi{MCB--graphs}. Since the tangential variety $\tau(X)$ is contained in the secant variety, the equations of the latter will also vanish on the former, and the coordinate ring of $\tau(X)$ will be a quotient of that of $\s_2(X)$. The MCB--graphs of distinct types are independent in (the generic version of) $\K[\s_2(X)]$, so our goal will be to identify the relations that they satisfy when passing to the quotient $\K[\tau(
X)]$.

\begin{definition}[MCB--graphs, {\cite[Section~4.2.3]{raiGSS}}] We say that an admissible graph $G$ (see Def. \ref{def:ideal_graph}) is \defi{maximally connected bipartite (MCB)} if it is either bipartite and connected,
or it is the union of a tree and a collection of isolated nodes. The \defi{type} $(a\geq b)$ of $G$ is a pair of integers representing the sizes of the sets $A,B$ in the bipartition of its maximal connected component (note the slight difference in terminology from \cite{raiGSS}, where the information of the shape of $G$ was part of the definition of its type; most of the time we will avoid referring to the shape of $G$, as it will be understood from the context). $G$ is \defi{canonically oriented} if all edges have endpoints in the smaller set of the bipartition. If $a=b$ then there are two canonical orientations.
\end{definition}

\begin{example}\label{ex:tabtograph} The graph $G$ below is a canonically oriented MCB--graph of type $(3,3)$ and shape $((5,2),(6,1),(5,2))$. We picture both the graph and the (an) associated Young tabloid.
 \[
\ytableausetup{boxsize=1.25em,tabloids,aligntableaux=center}
\xymatrixrowsep{.3pc}
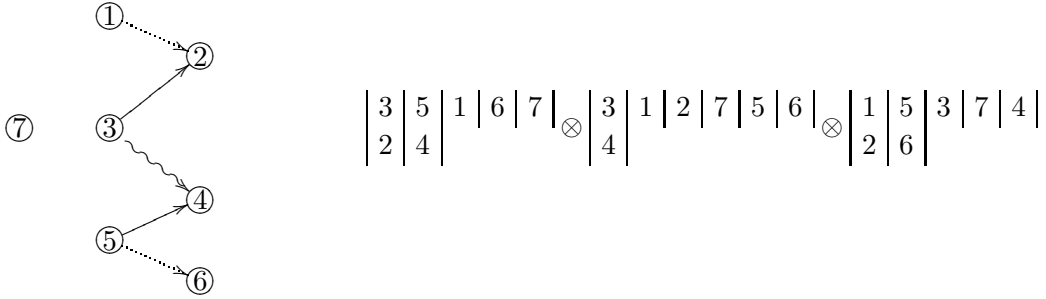
\xymatrix{
& \ccircle{1} \ar@{.>}[dr] & & & \\
& & \ccircle{2} & & \\
\ccircle{7} & \ccircle{3} \ar[ur] \ar@{~>}[dr] & &  & {\ytableaushort{35167,24}\oo\ytableaushort{312756,4}\oo\ytableaushort{15374,26}}\\
& & \ccircle{4} & & \\
& \ccircle{5} \ar[ur] \ar@{.>}[dr] & & & \\
& & \ccircle{6} & & \\
}
\]
The graph $G'$ obtained by reversing the directions of all arrows of $G$ is also a canonically oriented MCB--graph of the same type and shape as $G$. As a consequence of Proposition~\ref{prop:MCB}(4) below, $[G]=0$.
\end{example}

We collect a series of results on MCB--graphs that will be useful in the next section:

\begin{proposition}[{\cite[Section~4.2]{raiGSS}}]\label{prop:MCB} Fix a positive integer $r$ and let $\d,\r$ be as before. Consider an $n$--partition $\ll\vdash^n\r$ with each $\ll^i$ having at most two parts. With the usual identification of Young tabloids, graphs and covariants of $\SSS$, and letting $e_{\ll}=\sum_j\ll^j_2$ and $f_{\ll}=\max_j\lceil\ll_2^j/d_j\rceil$ (as in Theorem~\ref{thm:maincoordring}), we have: 
\begin{enumerate}
 \item There exists an MCB--graph of type $(a,b)$ and shape $\ll$ iff $b\geq f_{\ll}$ and $e_{\ll}\geq 2f_{\ll}-1$.

 \item $\rm{hwt}_{\ll}(\SSS)$ is spanned by MCB--graphs and graphs containing a triangle.

 \item Any two canonically oriented MCB--graphs of the same type and shape differ by a linear combination of graphs containing a triangle.

 \item An MCB--graph of type $(a,a)$ with an odd number of edges (such as the graphs in part (1) having $e_{\ll}=2f_{\ll}-1$) is a linear combination of graphs containing a triangle.

 \item Graphs containing odd cycles are linear combinations of graphs containing triangles.
\end{enumerate}
\end{proposition}

\section{A graphical description of the generic ideal of the tangential variety}\label{sec:graph-lanwey}

In this section we describe a collection of graphs that generate the ideal $\II\subset\SSS$ of generic equations of the tangential variety $\tau(X)$ of a Segre--Veronese variety $X$ (see Definition \ref{def:lanweyeqns}).

\begin{definition}\label{def:richgraph} We say that a graph $G$ is \defi{rich} if it has more edges than vertices. We write $\II^{rich}$ for the ideal of $\SSS$ generated by rich graphs.
\end{definition}

Our goal is to prove the following (recall the definition of $\II^{<3}$ from Example \ref{ex:subspace})

\begin{theorem}\label{thm:richgraphs}
 $\II=\II^{rich}+\II^{<3}$. Moreover, let $r>0$, $\ll\vdash^n\r$ and set
\[f_{\ll}=\max_{j=1,\cdots,n}\left\lceil\frac{\ll_2^j}{d_j}\right\rceil,\quad e_{\ll}=\ll^1_2+\cdots+\ll^n_2.\]
We have $m_{\ll}(\pi(\SSS))=0$ if some partition $\ll^j$ has more then two parts, or if $e_{\ll}<2f_{\ll}$, or if $e_{\ll}>r$, and $m_{\ll}(\pi(\SSS))=1$ otherwise.
\end{theorem}

\begin{remark} In terms of graphs, $r$ is the number of vertices, $e_{\ll}$ is the total number of edges, and $f_{\ll}$ is the maximum number of edges of a single color.
\end{remark}

It is an easy observation which we explain next that rich graphs lie in the kernel of $\pi$. Consider a rich graph $G$ and associate to it a Young$^{\d}_r$ tableau $F$ as in Definition \ref{def:ideal_graph}. $F$ has precisely $e=|E(G)|$ columns of size two, and $e>r=|V(G)|$. Consider an arbitrary partition $\ul{a}\vdash r$ and the associated map $\pi_{\ul{a}}$: we would like to prove that $\pi_{\ul{a}}([F])=0$. Note that $\pi_{\ul{a}}([F])$ is a sum of Young$^{\d,\ul{a}}_r$ tabloids $[F_i]$, where each $F_i$ has shape $\ll$ and entries in $\{1,2\}$, precisely $r$ of which are equal to $2$ (see Example \ref{ex:tableaux}). Since $r<e$, each $F_i$ contains a column equal to $\ytableaushort{1,1}$ and therefore $[F_i]=0$ by Lemma \ref{lem:fundrels}(a). The inclusion $\II^{<3}\subset\II$ is also immediate: if $F$ is a Young$^{\d}_r$ tabloid of shape $\ll$, with some $\ll^j$ having at least three parts, then $\pi([F])$ is a linear combination of tabloids $[F_i]$ of shape $\ll$ and entries in $\{1,2\}$. Any such $[F_
i]$ must then have repeated entries in any column of size at least $3$, so $[F_i]=0$ as before. Alternatively, $\II^{<3}\subset\II$ follows from the fact that the tangential variety $\t(X)$ is contained in the subspace variety described in Example \ref{ex:subspace}.

\begin{lemma}
If $G$ is a graph containing a rich subgraph then $[G]\in\II^{rich}$.
\end{lemma}

\begin{proof}
 Let $V\subset V(G)$ be a vertex subset supporting a rich subgraph. With the notation of Proposition \ref{prop:idealgraphs}, all the graphs in the family $\mc{G}$ are rich, so the conclusion follows.
\end{proof}

By the discussion at the beginning of Section \ref{subsec:MCB}, graphs containing triangles provide equations for $\s_2(X)$ so they belong to $\II$. The next step in proving Theorem~\ref{thm:richgraphs} is then naturally the following

\begin{lemma}\label{lem:triangle}
 If $G$ contains a triangle then $[G]$ belongs to the ideal $\II(\II^{rich}_3)$ generated by the degree $3$ part of $\II^{rich}$.
\end{lemma}

\begin{proof} Let $V\subset V(G)$ be a set of three vertices that supports a triangle, and let $\mc{G}$ be the associated family of graphs as in Proposition \ref{prop:idealgraphs}. All the graphs in $\mc{G}$ have $3$ vertices and contain a triangle, so it suffices to show that if $G$ is a graph with three vertices containing a triangle, then $[G]\in\II^{rich}_3$. If $G$ has at least four edges then it is a rich graph, so we may assume that $G$ has precisely three edges, i.e. it is a triangle. We will show that in this case $[G]=0$. To do that we have to show that if $T$ is any Young$^{\d}_3$ tabloid containing one of
\begin{equation}\label{eq:triangles}
\ytableausetup{boxsize=1.25em,tabloids,aligntableaux=center} \ytableaushort{123,231},\quad \ytableaushort{1123,23}\oo\ytableaushort{21,3},\quad \ytableaushort{13,2}\oo\ytableaushort{21,3}\oo\ytableaushort{32,1}
\end{equation}
and possibly some other columns of size one, then $T=0$. Recall that $T$ is unchanged by permutations of the labels $1,2,3$, as well as by permutations of its columns, and it changes sign when applying a transposition within any of its columns (Lemma \ref{lem:fundrels}(a)). We get
\[\ytableausetup{boxsize=1.25em,tabloids,aligntableaux=center} \ytableaushort{123,231}\overset{1\leftrightarrow 2}{=}\ytableaushort{213,132}\overset{\substack{\text{swap}\\ \text{rows}}}{=}(-1)^3\cdot \ytableaushort{132,213}\overset{\substack{\text{swap last}\\ \text{two cols}}}{=}-\ytableaushort{123,231}\ ,\]
\[\ytableausetup{boxsize=1.25em,tabloids,aligntableaux=center} \ytableaushort{1123,23}\oo\ytableaushort{21,3}\overset{2\leftrightarrow 3}{=}\ytableaushort{1132,32}\oo\ytableaushort{31,2}=-\ytableaushort{1132,32}\oo\ytableaushort{21,3}\]
\[=-\ytableaushort{1123,23}\oo\ytableaushort{21,3},\]
showing that the first two Young tabloids in (\ref{eq:triangles}) satisfy $T=-T$, i.e. $T=0$. For the last one
\[\ytableausetup{boxsize=1.25em,tabloids,aligntableaux=center}\ytableaushort{13,2}\oo\ytableaushort{21,3}\oo\ytableaushort{32,1}\overset{\text{Lemma~}\ref{lem:fundrels}}{=}\ytableaushort{13,2}\oo\left(\ytableaushort{12,3}+\ytableaushort{23,1}\right)\oo\left(\ytableaushort{23,1}+\ytableaushort{31,2}\right)\]
\[\ytableausetup{boxsize=1.25em,tabloids,aligntableaux=center} = \ytableaushort{13,2}\oo\ytableaushort{12,3}\oo\ytableaushort{23,1}+\ytableaushort{13,2}\oo\ytableaushort{12,3}\oo\ytableaushort{31,2}\]
\[\ +\ \ytableaushort{13,2}\oo\ytableaushort{23,1}\oo\ytableaushort{23,1}+\ytableaushort{13,2}\oo\ytableaushort{23,1}\oo\ytableaushort{31,2}.\]
We write $A,B,C,D$ (in this order) for the four Young tabloids above. We have $B=-T$ ($1\leftrightarrow 2$), and $C=-C$ ($1\leftrightarrow 2$), so $C=0$. Applying Lemma~\ref{lem:fundrels}(b) to the first factor of $D$ we get
\[\ytableausetup{boxsize=1.25em,tabloids,aligntableaux=center} D = \ytableaushort{31,2}\oo\ytableaushort{23,1}\oo\ytableaushort{31,2} + \ytableaushort{12,3}\oo\ytableaushort{23,1}\oo\ytableaushort{31,2}= -A - T
\]
where the last equality follows by applying the permutation $1\to 3\to 2\to 1$ to the first tabloid, and $1\to 2\to 3\to 1$ to the second. Putting everything together, it follows that $T = A+B+C+D = A + (- T) + 0 + (-A - T)$, i.e. $T=-2T$ and thus $T=0$. 
\end{proof}

Before proving Theorem \ref{thm:richgraphs}, we need some more terminology: recall that if $G$ is a graph as in Definition \ref{def:ideal_graph}, $\ccirc{x}\in V(G)$, and $j\in[r]$, we say that $\ccirc{x}$ is \defi{$j$--saturated} if it has $d_j$ incident edges of color $j$. The \defi{degree} of a node $\ccirc{x}$ is the number of incident edges. A node of degree one is called a \defi{leaf}.

\begin{proof}[Proof of Theorem \ref{thm:richgraphs}] Fix $r$ and $\ll\vdash^n\r$. Since $\II^{<3}\subset\II$, we may assume that all $\ll^j$ have at most two parts. We write $\pi_{\ll}$ for the restriction of $\pi$ to a map $\rm{hwt}_{\ll}(\SSS)\to \rm{hwt}_{\ll}(\SSS')$. We show that if $e_{\ll}<2 f_{\ll}$ or $e_{\ll}>r$, then $\pi_{\ll}(\rm{hwt}_{\ll}(\SSS))=0$. When $2f_{\ll}\leq e_{\ll}\leq r$, we construct in Lemma \ref{lem:mlgeq1} a graph $G$ of shape $\ll$ with $\pi_{\ll}([G])\neq 0$ from which we deduce that $m_{\ll}(\pi(\SSS))\geq 1$. Moreover, we show that $\rm{hwt}_{\ll}(\SSS)$ is spanned by graphs contained in $\II^{rich}$ together with MCB--graphs of a single type, which implies that $m_{\ll}(\pi(\SSS))\leq 1$ (recall that by Proposition \ref{prop:MCB}(3) and Lemma \ref{lem:triangle} any two MCB--graphs of the same type span the same vector space modulo $\II^{rich}$). Since $\II^{rich}\subset\II$ we obtain that $\II_{\ll}=\II^{rich}_{\ll}$ and $m_\ll(\pi(\SSS))=1$, as desired.

If $e_{\ll}<2f_{\ll}$, then $S_{\ll}V$ does not occur in the coordinate ring of $\s_2(X)$ \cite[Theorem~4.1]{raiGSS}, so it can't occur in the coordinate ring of $\tau(X)$ either, therefore $\pi(\SSS)_{\ll}=0$ (see Example \ref{ex:multiplicities}). If $e_{\ll}>r$ then any graph $G$ of shape $\ll$ is rich, so $\II_{\ll}=\SSS_{\ll}$ and $\pi(\SSS)_{\ll}=0$ again. We may then assume that $2f_{\ll}\leq e_{\ll}\leq r$. By Lemma \ref{lem:mlgeq1} below, $m_{\ll}(\pi(\SSS))\geq 1$ in this case. We show that MCB--graphs of a single type suffice to generate $\rm{hwt}_{\ll}(\SSS_{\A})$ modulo $\II^{rich}$. We distinguish two cases.

\subsection{Case $e_{\ll}=r$.}\label{case:e=r}  We show that any MCB--graph is (up to a factor) congruent modulo $\II^{rich}$ with one of type $(a,b)$ with $b\leq a\leq b+1$. Once this is shown it follows that $\rm{hwt}_{\ll}(\SSS)$ is spanned modulo $\II^{rich}$ by MCB--graphs of type $(\lceil r/2\rceil,\lfloor r/2\rfloor)$. Let $G$ be an MCB--graph of type $(a,b)$ with $a\geq b+2$, and write $A,B$ for the two sets of vertices of its bipartition. Observe first that there exists a leaf $\ccirc{x}\in A$: otherwise, each node in $A$ would have degree at least $2$, so $G$ would have at least $2a\geq a+b+2=r+2$ edges, but $e_{\ll}\leq r$, a contradiction. Let $j$ be the color of the unique edge $(x,y)$ incident to $\ccirc{x}$. We now claim that there is a node $\ccirc{x}\neq\ccirc{z}\in A$ which is not $j$--saturated: if this were false then there would be a total of $(a-1)d_j+1>bd_j$ edges of color $j$, but this is impossible since there are at most $bd_j$ edges of color $j$ incident to the vertices in $B$. 
Consider the 
relation (\ref{eq:3plucker}) with the given $\ccirc{x},\ccirc{y}$ and $\ccirc{z}$. This expresses $[G]=[G_1]+[G_2]$ where $G_1$ is an MCB--graph of type $(a-1,b+1)$ and $G_2$ is a graph containing a rich subgraph, hence $[G]\equiv [G_1]$ modulo $\II^{rich}$. We conclude by induction.

\subsection{Case $e_{\ll}<r$.}\label{case:e<r} We show that any MCB--graph is congruent modulo $\II^{rich}$ with one of type $(a,b)$ with $b\leq a\leq b+2$: since $a+b=e_{\ll}+1$, Proposition \ref{prop:MCB}(4) shows that any MCB--graph of type $(a,a)$ is in fact in $\II^{rich}$, so the condition $a\leq b+2$ yields a unique possibility for a type of MCB--graph not in $\II^{rich}$, namely $(\lceil e_{\ll}/2\rceil+1,\lfloor e_{\ll}/2\rfloor)$. Let $G$ be an MCB--graph of type $(a,b)$ with $a\geq b+3$, and write $A,B$ for the two sets of vertices of its bipartition. By the same arguments as before we can find $\ccirc{x},\ccirc{y},\ccirc{z}$ as in Case \ref{case:e=r}. The relation (\ref{eq:3plucker}) expresses $[G]=[G_1]+[G_2]$ where $G_1$ is an MCB--graph of type $(a-1,b+1)$ and $G_2$ is a union of isolated nodes (including $\ccirc{x}$) and a connected component $G_2'$ which is an MCB--graph of type $(a-1,b)$ with the same number of edges as vertices ($a+b-1$). Since $a-1\geq b+2$, we are in the situation of 
Case \ref{case:e=r}. We get that $[G_2]\equiv [G_3]$ modulo $\II^{rich}$ where $G_3$ is the union of the isolated nodes in $G_2\setminus G_2'$ together with a connected component $G_3'$ which is an MCB--graph of type $(a-2,b+1)$. Now let $(u,v)$ be an edge contained in the unique cycle of $G_3'$ (say $\ccirc{u}\in A$, $\ccirc{v}\in B$), and consider the relation (\ref{eq:3plucker}) with $x,y,z$ replaced by $u,v,x$ respectively. This expresses $[G_3]=[G_4]+[G_5]$ where $G_4$ is an MCB--graph of type $(a-2,b+2)$ (or $(b+2,a-2)$ if $a=b+3$) and $G_5$ is an MCB--graph of type $(a-1,b+1)$. Putting everything together, we get
\[[G]=[G_1]+[G_2]\overset{\II^{rich}}{\equiv} [G_1]+[G_3]=[G_1]+[G_4]+[G_5],\]
where $G_1,G_4,G_5$ are MCB--graphs of type $(a',b')$ with $a'-b'<a-b$. We conclude again by induction.
\end{proof}

We end with the promised

\begin{lemma}\label{lem:mlgeq1}
 If $2f_{\ll}\leq e_{\ll}\leq r$, then there exists a partition $\ul{a}\vdash r$ and a graph $G$ (or equivalently an $n$--tableau $T$) of shape $\ll$, such that $\pi_{\ul{a}}(G)\neq 0$.
\end{lemma}

\begin{proof}
 Let $m=\lfloor e_{\ll}/2\rfloor$ and consider the Young$^{\d}_r$ tabloid $T$ whose columns of size two are (from left to right)
\begin{equation}\label{eq:2cols}
\ytabonecol{1}{2},\ytabonecol{3}{4},\cdots,\ytabonecol{2m-1}{2m},\ytabonecol{1}{2},\ytabonecol{3}{4},\cdots,\ytabonecol{2m-1}{2m},(\text{and }\ytabonecol{1}{2m+1}) 
\end{equation}
where the last column appears when $e_{\ll}=2m+1$ is odd. For example, when $\d=(1,2,1)$, $r=8$ and $e_{\ll}=7$, $T$ could be
\begin{equation}\label{eq:someT}
\ytableausetup{boxsize=1.25em,tabloids,aligntableaux=center} \ytableaushort{135678,24}\oo\ytableaushort{513512347788,6246}\oo\ytableaushort{1234568,7}.
\end{equation}

To check that $T$ is a valid Young tabloid, we need to show that no $i\in[r]$ appears more than $d_j$ times within the columns of size two of $T^j$. If $d_j=1$ and some $i\in[r]$ occurs twice in $T^j$, we must have $\ll^j_2\geq m+1$, so $f_{\ll}\geq\ll^j_2\geq m+1$, but $2f_{\ll}\leq e_{\ll}\leq 2m+1$, a contradiction. If $d_j=2$ then the only $i\in[r]$ that could appear thrice in $T^j$ would be $1$, in which case all columns in (\ref{eq:2cols}) would belong to $T^j$, so $\ll^j_2=e_{\ll}=2m+1$. But then $f_{\ll}=\lceil(2m+1)/2\rceil=m+1$ and we get a contradiction as before.

We define the partition $\ul{a}\vdash r$ by letting $a_j=\ll^j_2$ for $j>1$, and $a_1=r-(a_2+\cdots+a_n)$. We will show that $\pi_{\ul{a}}(T)\neq 0$. In fact
\begin{equation}\label{eq:piaT}
\pi_{\ul{a}}(T)=2^m\cdot(-1)^m\cdot d_1^{r-e_{\ll}}\cdot \tilde{T}, 
\end{equation}
where $\tilde{T}$ is the Young tabloid obtained by listing in each $\tilde{T}^j$ the $(rd_j-a_j)$ ones followed by the $a_j$ twos, aligned from left to right and top to bottom. Note that $\tilde{T}$ gives a basis for the $1$--dimensional vector space $\rm{hwt}_{\ll}(\SSS'_{\ul{a}})$). For $T$ in (\ref{eq:someT}) we have
\[\ytableausetup{boxsize=1.25em,tabloids,aligntableaux=center} \pi_{\ul{a}}(T)=-8\cdot\ytableaushort{111112,22}\oo\ytableaushort{111111111111,2222}\oo\ytableaushort{1111111,2}.\]
To understand (\ref{eq:piaT}), note that $\pi_{\ul{a}}(T)$ is a sum of Young tabloids $T'$ obtained from $T$ by selecting one of each of the entries $1,\cdots,r$ (in such a way that $a_j$ entries are selected from $T^j$) and replacing them by $2$, while making all other entries of $T$ equal to $1$ (see Example~\ref{ex:tableaux}). All the Young tabloids $T'$ obtained in this way that have repeated entries in some column are zero by Lemma \ref{lem:fundrels}(a), so we disregard them. The only interesting Young tabloids $T'$ that we obtain are therefore the ones whose columns of size two are $\ytabonecol{1}{2}$ or $\ytabonecol{2}{1}$. It follows that if one of the $\ytabonecol{2i-1}{2i}$ columns of $T$ becomes $\ytabonecol{1}{2}$ in $T'$ the other one must become $\ytabonecol{2}{1}$ and vice versa (this yields a total of $2^m$ choices). There is no choice for the column $\ytabonecol{1}{2m+1}$: it must become $\ytabonecol{1}{2}$ in $T'$ since already one of the two $1$s from the two columns $\ytabonecol{1}{2}$ of 
$T$ has been replaced by a $2$. To understand the factor $d_1^{r-e_{\ll}}$ in (\ref{eq:piaT}), note that each of the $e_{\ll}$ twos in the columns of size two of $T'$ correspond to one of $1,2,\cdots,e_{\ll}$ from $T$. This means that we are left with the choice of selecting one of each of the occurrences of $e_{\ll}+1,\cdots,r$ from $T^1$, and this can be done in $d_1^{r-e_{\ll}}$ ways. Finally, the $(-1)^m$ is explained by the fact that in order to get from any $T'$ to $\tilde{T}$ one needs to perform a transposition in each of the columns $\ytabonecol{2}{1}$ of $T'$ (possibly followed by some permutation of the columns of size one of $T'$).
\end{proof}

\section{Minimal generators of the ideal of the tangential variety}\label{sec:mingensI}

In this section we determine the minimal generators of the ideal of the tangential variety of an arbitrary Segre--Veronese variety and obtain as a special case a confirmation of the Landsberg--Weyman conjecture (Conjecture \ref{conj:lanwey}). We will work entirely in the generic setting. We write $\II(q)=\II(\II_q)$ for the ideal in $\SSS$ generated by $\II_q$. The degree $(q+1)$ part of $\II(q)$ is a subfunctor of $\II_{q+1}$, so we can define the \defi{syzygy functor} $\KK_{1,q}=\II_{q+1}/\II(q)_{q+1}$ of minimal generators of degree $(q+1)$ of the ideal of generic equations of the tangential variety. We prove that $\KK_{1,q}=0$ for $q>3$ and give a precise description of a minimal set of generators of $\KK_{1,q}$ for $q=1,2,3$.

We start with the description of $\KK_{1,1}=\II_2$ (note that $\II(1)=0$: $\tau(X)$ is nondegenerate, so there are no linear forms in its ideal). Since $\II_2=\II_2^{rich}$ (Theorem \ref{thm:richgraphs}), $\KK_{1,1}$ is generated by rich graphs on two vertices. These are MCB--graphs of type $(1,1)$ and are nonzero precisely when they have an even number of edges. We get
\[m_{\ll}(\KK_{1,1})=
\begin{cases}
1 & \text{if each }\ll^j\text{ has at most two parts and }\sum_j\ll^j_2>2\text{ is even,}\\
0 & \text{otherwise.}
\end{cases}
\]

\subsection{Proof that $I(\tau(X))$ is generated in degree at most $4$}\label{subsec:Ileq4}
Note first that by Lemma~\ref{lem:triangle} and Proposition~\ref{prop:MCB}(5) graphs that contain an odd cycle belong to the ideal $\II(\II^{rich}_3)$ generated by rich graphs with three vertices. Consider an arbitrary rich graph $H$ without odd cycles: we would like to prove that $H$ is in the ideal $\II(\II^{rich}_{\leq 4})$ generated by rich graphs with at most $4$ vertices. If $H$ contains a rich subgraph with $4$ vertices, then the conclusion follows from Proposition \ref{prop:idealgraphs}. In fact, consider a minimal subgraph $G$ of $H$ that is rich: $G$ is connected and $|E(G)|=|V(G)|+1$, so it contains precisely two (independent) cycles. If we can show using only the Pl\"ucker relations (\ref{eq:3plucker}) and (\ref{eq:plucker}) that $G$ is a linear combination of graphs containing rich subgraphs with at most $4$ vertices, then we can lift those relations to $H$ by Remark \ref{rem:subgraphs} to conclude that the same is true about $H$. Proposition \ref{prop:idealgraphs} then applies to show that $[H]\
in\II(\II^{rich}_{\leq 4})$. This will be our proof strategy.

For most of this section we disregard the orientation of the arrows since it's irrelevant for our arguments. Our conclusions should then be interpreted ``up to sign''. We start with

\begin{lemma}\label{lem:star}
 Suppose that $G$ is a graph containing
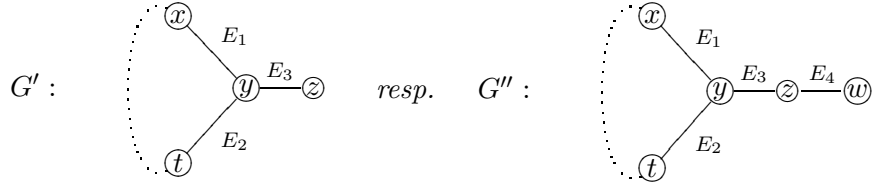
\[\xymatrix@=15pt{
 & \ccircle{x} \ar@(ul,dl)@{..}[dd]  \ar@{-}[dr]^{E_1} & & \\
 G':\quad\quad &  & \ccircle{y} \ar@{-}[r]^{E_3} & \ccircle{z} \\
 & \ccircle{t} \ar@{-}[ur]_{E_2} & & \\
}
\xymatrix@=15pt{
 & \ccircle{x} \ar@(ul,dl)@{..}[dd]  \ar@{-}[dr]^{E_1} & & & \\
 \quad\text{ resp. }\quad G'':\quad\quad &  & \ccircle{y} \ar@{-}[r]^{E_3} & \ccircle{z} \ar@{-}[r]^{E_4} & \ccircle{w}\\
 & \ccircle{t} \ar@{-}[ur]_{E_2} & & & \\
}
\]
where $(x,y)$ and $(y,t)$ are consecutive edges in an even cycle (possibly $x=t$). We have that $[G]\equiv[\tilde{G}]$ modulo $\II(\II^{rich}_3)$, where $\tilde{G}$ contains the subgraph $\tilde{G}'$ (resp. $\tilde{G}''$) obtained from $G'$ (resp. $G''$) by replacing any (resp. one) of the edges $E_1$ and $E_2$ with a new edge between $\ccirc{y}$ and $\ccirc{z}$. Moreover, one can insure that there is at most one edge of $G$ outside $G'$ (resp. $G''$) that is not preserved when going from $G$ to $\tilde{G}$, and that that edge is incident to $\ccirc{z}$.
\end{lemma}

\begin{proof} We prove the lemma in the case when $G$ contains $G''$, and we mention at the end the changes that are needed to prove the statement for $G'$. Write $c_i$ for the color of the edge $E_i$. We claim that $\ccirc{z}$ is not $c_i$--saturated as a vertex of $G''$ for some $i=1,2$ (note that $\ccirc{z}$ could be $c_i$--saturated as a vertex of $G$ for both $i=1,2$). If this wasn't true, one of the following alternatives would hold:
\begin{itemize}
 \item $c_1=c_2=c$, in which case we must have $d_c\geq 2$. Since $\ccirc{z}$ is $c$--saturated, we must have $d_c=2$ and both edges $E_3$ and $E_4$ incident to $\ccirc{z}$ would have to have color $c$. In particular, $E_1,E_2,E_3$ would be $3>d_c$ edges of color $c$ incident to $\ccirc{y}$, a contradiction.

 \item $c_1\neq c_2$, in which case the two edges $E_3$ and $E_4$ incident to $\ccirc{z}$ would have to have colors $c_1$ and $c_2$ (in some order), and moreover, $d_{c_1}=d_{c_2}=1$. We may assume that $c_3=c_1$ in which case $E_1$ and $E_3$ are $2>d_{c_1}$ edges of color $c_1$ incident to $\ccirc{y}$, a contradiction.
\end{itemize}
We may then assume that $\ccirc{z}$ is not $c_1$--saturated, and use relation (\ref{eq:3plucker}). This expresses $[G'']=[G''_1]+[\tilde{G}'']$ where $G_1''$ contains an odd cycle, so $[G_1'']\in\II(\II^{rich}_3)$.
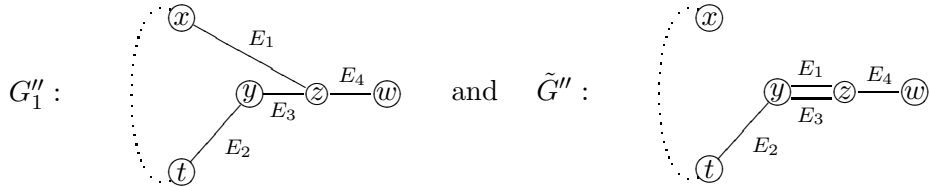
\[\xymatrix@=15pt{
 & \ccircle{x} \ar@(ul,dl)@{..}[dd]  \ar@{-}[drr]^{E_1} & & & \\
 G''_1:\quad\quad &  & \ccircle{y} \ar@{-}[r]_{E_3} & \ccircle{z} \ar@{-}[r]^{E_4} & \ccircle{w}\\
 & \ccircle{t} \ar@{-}[ur]_{E_2} & & & \\
}
\xymatrix@=15pt{
 & \ccircle{x} \ar@(ul,dl)@{..}[dd]  & & & \\
 \quad\text{ and }\quad \tilde{G}'':\quad\quad &  & \ccircle{y} \ar@<.5ex>@{-}[r]^{E_1} \ar@<-.5ex>@{-}[r]_{E_3} & \ccircle{z} \ar@{-}[r]^{E_4} & \ccircle{w}\\
 & \ccircle{t} \ar@{-}[ur]_{E_2} & & & \\
} 
\]
We lift the relation $[G'']=[G''_1]+[\tilde{G}'']$ (as in Remark \ref{rem:subgraphs}) to a relation $[G]=[G_1]+[\tilde{G}]$ where $[G_1]\in\II(\II^{rich}_3)$ and $\tilde{G}$ contains $\tilde{G}''$, which yields the desired conclusion. To explain the last statement of the lemma, note that if $(z,w')$ is an edge of color $c_1$ in $G$ which when added to $G''$ makes $\ccirc{z}$ a $c_1$--saturated vertex, then when lifting the relation $[G'']=[G''_1]+[\tilde{G}'']$ to $[G]=[G_1]+[\tilde{G}]$, we will have to replace the edge $(z,w')$ by $(x,w')$.

In the case when $G$ contains $G'$, note that $\ccirc{z}$ is not $c_i$--saturated as a vertex of $G'$ for any $i=1,2$ (otherwise $c_3=c_i$, $d_{c_i}=1$ and $E_i,E_3$ are two edges of color $c_i$ incident to $\ccirc{y}$). We can then apply relation (\ref{eq:3plucker}) with $\ccirc{z}$ and any of the edges $E_1$, $E_2$, and conclude as before.
\end{proof}

Going back to our graph $H$ and its minimal rich subgraph $G$, we show that we may assume that $G$ contains a $2$--cycle. Consider a cycle $C$ in $G$ of minimal (even) length $l(C)$, and assume that $l(C)>2$. Since $C$ has minimal length, and since $G$ is connected and contains a cycle other than $C$, there must be an edge $(y,z)$ in $G$ with $\ccirc{y}\in C$ and $\ccirc{z}$ outside $C$. Letting $G'$ be the subgraph consisting of $C$ together with the edge $(y,z)$ we can apply Lemma \ref{lem:star} to reduce to the case when $G$ contains a $2$--cycle $C_0$. Consider a second cycle $C\neq C_0$ in $G$ and assume for now that $C$ and $C_0$ are disjoint (i.e. have no common vertices).

\begin{lemma}\label{lem:distantcycles}
 Assume that $G$ is a graph containing disjoint cycles $C_0$ and $C$, with $l(C_0)=2$. Let $P=yz\cdots u$ be a path joining a vertex of $C_0$ to one of $C$. We call $l(P)$ the \defi{distance} between $C$ and $C_0$.
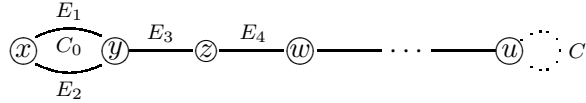
\[\quad
\xymatrix{
\ccircle{x} \ar@<-.5ex>@/_/@{-}[r]_{E_2} \ar@<.5ex>@/^/@{-}[r]^{E_1}_{C_0} & \ccircle{y} \ar@{-}[r]^{E_3} & \ccircle{z} \ar@{-}[r]^{E_4} & \ccircle{w} \ar@{-}[r] & \cdots \ar@{-}[r] & \ccircle{u} \ar@(ur,dr)@{..}^C \\
}
\]
We have $[G]\equiv[\tilde{G}]$ modulo $\II(\II^{rich}_3)$, where $\tilde{G}$ contains $C$ and a $2$--cycle $C_0'$ at distance $1$ from $C$.
\end{lemma}

\begin{proof}
 Use Lemma \ref{lem:star} to slide the $2$--cycle $C_0$ down the path $P$: namely if $l(P)>1$, apply the said lemma with $t=x$ and $G''$ the subgraph consisting of the edges $E_1,\cdots,E_4$.
\end{proof}
\noindent By Lemma \ref{lem:distantcycles}, we may assume that the distance between $C$ and $C_0$ is $1$, i.e. $G$ contains
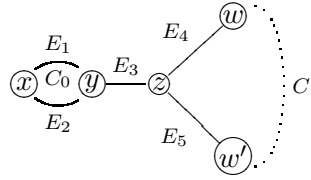
\[
\xymatrix@=15pt{
  & & & \ccircle{w}  \ar@(ur,dr)@{..}[dd]^C \\
  \ccircle{x} \ar@<-.5ex>@/_/@{-}[r]_{E_2} \ar@<.5ex>@/^/@{-}[r]^{E_1}_{C_0} & \ccircle{y} \ar@{-}[r]^{E_3} & \ccircle{z} \ar@{-}[ur]^{E_4} \ar@{-}[dr]_{E_5} & \\
  & & & \ccircle{w'} \\
} 
\]
We can still apply Lemma \ref{lem:star} for $G$ containing the subgraph $G''$ determined by the edges $E_1,\cdots,E_4$. The argument in the proof of the said lemma implies that $\tilde{G}$ must be one of
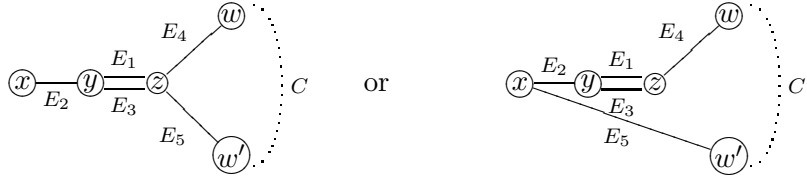
\[\xymatrix@=15pt{
 & & & \ccircle{w}  \ar@(ur,dr)@{..}[dd]^C \\
 \ccircle{x} \ar@{-}[r]_{E_2} & \ccircle{y} \ar@<-.5ex>@{-}[r]_{E_3} \ar@<.5ex>@{-}[r]^{E_1} & \ccircle{z} \ar@{-}[ur]^{E_4} \ar@{-}[dr]_{E_5}  \\
 & & & \ccircle{w'} \\
}
\xymatrix@=15pt{
 & & & & \ccircle{w}  \ar@(ur,dr)@{..}[dd]^C \\
\quad\text{ or }\quad\quad & \ccircle{x} \ar@{-}[drrr]_{E_5} \ar@{-}[r]^{E_2} & \ccircle{y} \ar@<-.5ex>@{-}[r]_{E_3} \ar@<.5ex>@{-}[r]^{E_1} & \ccircle{z} \ar@{-}[ur]^{E_4}   \\
 & & & & \ccircle{w'} \\
}
\]
(the latter situation occurs for example when $E_1,E_2,E_4,E_5$ have the same color $c$ and $d_c=2$.) We are then reduced to the case when $C$ and $C_0$ intersect: in the first case they have a common vertex, while in the second they share a common edge. The latter case is treated in Lemma \ref{lem:doublecycle} below, and we show that the former cases reduces to it. Assume that $G$ contains
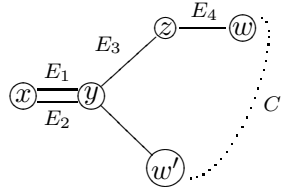
\[
\xymatrix@=15pt{
 & & \ccircle{z} \ar@{-}[r]^{E_4} & \ccircle{w} \ar@(ur,dr)@{..}[ddl]^C \\
\ccircle{x} \ar@<-.5ex>@{-}[r]_{E_2} \ar@<.5ex>@{-}[r]^{E_1} & \ccircle{y} \ar@{-}[ur]^{E_3} \ar@{-}[dr] & \\
 & & \ccircle{w'} \\
} 
\]
Applying again Lemma \ref{lem:star} with $x=t$ and $G''$ the subgraph determined by $E_1,\cdots,E_4$ (note that now $E_3$ and $E_4$ are the only edges in $C$ incident to $\ccirc{z}$) we are reduced to the situation of Lemma \ref{lem:doublecycle}.

\begin{lemma}\label{lem:doublecycle} 
 If $H$ is a graph containing a subgraph $G_0$ that consists of a $2$--cycle $C_0$ and an even cycle $C$ intersecting along a common edge, then $[H]\in\II(\II^{rich}_{\leq 4})$.
\end{lemma}

\begin{proof} If $H$ contains two $2$--cycles in the same connected component, then by Lemma \ref{lem:distantcycles} we may assume that these cycles are at distance at most $1$ from each other, so they determine a rich subgraph of $H$ with at most $4$ vertices and we can conclude using Proposition \ref{prop:idealgraphs}.
 
Assume now that $H$ doesn't contain two $2$--cycles joined by a path, and take $C$ to be a minimal (even) cycle sharing an edge with $C_0$. If $G_0=C\cup C_0$ is not a connected component of $H$, then we can find an edge $E_1=(x,y)\in E(C)\setminus E(C_0)$, and an edge $E_3=(y,z)\in E(H)$ with $\ccirc{z}\notin C$. Applying Lemma \ref{lem:star} we reduce to the case when $H$ contains two $2$--cycles joined by a path.

Finally, assume that $G_0$ is a connected component of $H$. By the previous paragraph, if $G$ is any connected graph that strictly contains $G_0$, then $[G]\in\II(\II^{rich}_{\leq 4})$. In order to show that $[H]\in\II(\II^{rich}_{\leq 4})$ it is then sufficient by Proposition \ref{prop:idealgraphs} to show that $[G_0]\in\II(\II^{rich}_{\leq 4})$. Writing $|V(G_0)|=2a$ we see that $G_0$ is an MCB--graph of type $(a,a)$ with an odd number of edges (namely $2a+1$) so by Proposition \ref{prop:MCB}(4) and Lemma \ref{lem:triangle}, we get that $[G_0]\in\II(\II^{rich}_3)$ concluding the proof.
\end{proof}

Throughout the rest of the paper, when referring to a Young tabloid, we will only illustrate its columns of size larger than $1$, the rest being determined (up to a permutation) by those and by $\d$.

\subsection{Minimal generators of degree $3$}\label{subsec:mingens3}
 In this section we describe the minimal generators of degree $3$ of the generic ideal $\II$ of $\t(X)$.
\begin{theorem}\label{thm:mingens3} For $\ll\vdash^n 3\d$, we have that $m_{\ll}(\KK_{1,2})\in\{0,1\}$ for all $\ll$, with (essentially) one exception: if $\ll\vdash^n 3\d$ has the property that $\ll^j=(3d_j-1,1)$ for exactly four values of $j$, and $\ll^j=(3d_j)$ for the rest, then $m_{\ll}(\KK_{1,2})=2$. More precisely, the following Young tabloids (classified according to the number of columns of size larger than $1$) give a minimal list of generators of the syzygy functor $\KK_{1,2}$ (note that some of them only occur for certain values of $n$ and $\d$, as specified; for simplicity, we state the results only up to a permutation of $d_1,\cdots,d_n$).

\begin{enumerate}
 \item Two columns of size $3$: 
\begin{equation}\label{eq:two3}
\ytableausetup{boxsize=1.25em,tabloids,aligntableaux=center}
  \ytableaushort{11,22,33}\quad\textrm{ or }\quad\ytableaushort{1,2,3}\oo\ytableaushort{1,2,3},
\end{equation}
where $d_1\geq 2$ for the first tabloid, and $n\geq 2$ for the second.

 \item One column of size $3$.
\begin{itemize}
 \item Two columns of size $2$:
\begin{equation}\label{eq:one3two2}
\ytableausetup{boxsize=1.25em,tabloids,aligntableaux=center}
  \ytableaushort{11,22,3}\oo\ytableaushort{1,3}\quad\textrm{ or }\quad\ytableaushort{1,2,3}\oo\ytableaushort{1,2}\oo\ytableaushort{1,3},
\end{equation}
where $n,d_1\geq 2$ for the first tabloid, and $n\geq 3$ for the second.

 \item Three columns of size $2$:
\begin{equation}\label{eq:one3three2}
\ytableausetup{boxsize=1.25em,tabloids,aligntableaux=center}
  \ytableaushort{112,233}\oo\ytableaushort{1,2,3}\quad\textrm{ or }\quad\ytableaushort{1112,2233,3},
\end{equation}
where $n\geq 2$, $d_1=2$ for the first tabloid, and $d_1=3$ for the second.
\end{itemize}

 \item No column of size $3$.
\begin{itemize}
 \item Four columns of size $2$:
\begin{equation}\label{eq:zero3four2grps}
\ytableausetup{boxsize=1.25em,tabloids,aligntableaux=center}
  \ytableaushort{11,23}\oo\ytableaushort{11,23}\quad\textrm{ or }\quad\ytableaushort{11,23}\oo\ytableaushort{1,2}\oo\ytableaushort{1,3}
\end{equation}
\begin{equation}\label{eq:zero3four2sings}
\ytableausetup{boxsize=1.25em,tabloids,aligntableaux=center}
  \textrm{ or }\quad\ytableaushort{1,2}\oo\ytableaushort{1,3}\oo\ytableaushort{1,2}\oo\ytableaushort{1,3}\quad\textrm{ or }\quad\ytableaushort{1,2}\oo\ytableaushort{1,3}\oo\ytableaushort{1,3}\oo\ytableaushort{1,2},
\end{equation}
where $n\geq 2$ for the first tabloid, $n\geq 3$ for the second, and $n\geq 4$ for the third and fourth.
 \item Five columns of size $2$:
\begin{equation}\label{eq:zero3five2}
\ytableausetup{boxsize=1.25em,tabloids,aligntableaux=center}
  \ytableaushort{1121,2332}\oo\ytableaushort{1,3}\quad\textrm{ or }\quad\ytableaushort{112,233}\oo\ytableaushort{1,2}\oo\ytableaushort{1,3},
\end{equation}
where $n\geq 2$, $d_1=3$ for the first tabloid, and $n\geq 3$, $d_1=2$ for the second.
 \item Six columns of size $2$:
\begin{equation}\label{eq:zero3six2}
\ytableausetup{boxsize=1.25em,tabloids,aligntableaux=center}
  \ytableaushort{112,233}\oo\ytableaushort{112,233}\quad\textrm{ or }\quad\ytableaushort{112112,233233},
\end{equation}
where $n\geq 2$, $d_1=d_2=2$ for the first tabloid, and $d_1=4$ for the second.
\end{itemize}
\end{enumerate}
\end{theorem}

We start with a consequence of Theorem \ref{thm:idealdtab} that will be used repeatedly in what follows. If $T$ is a Young$^{\d}_r$ tableau, we call a subset $\mc{C}$ of its columns \defi{initial} if for every $j\in[n]$, the columns of $T^j$ contained in $\mc{C}$ are the first $t_j$ columns of $T^j$ for some $t_j\geq 0$.

\begin{lemma}\label{lem:three12s}
 Fix positive integers $x,y\in[r]$. If $F$ is a Young$^{\d}_r$ tableaux of shape $\ll$ that contains an initial set $\mc{C}$ of columns with $|\mc{C}|\geq 3$, such that each column of $\mc{C}$ contains the entries $x,y$, then $[F]\in\II(\II^{rich}_2)=\II(2)$.
\end{lemma}

\begin{proof}
 Since $[F]=[\s\circ F]$ for $\s\in\SS_r$, we may assume that $x=1$ and $y=2$. If $F^{-1}([2])$ is not a Young diagram, we can apply the straightening relations from Lemma~\ref{lem:fundrels} repeatedly and write $[F]$ as a linear combination of $[F_i]$, where each $F_i$ contains the initial set of columns $\mc{C}$, and moreover $F_i^{-1}([2])=D_{\delta_i}$ for some $\delta_i\subset\ll$. Now Theorem~\ref{thm:idealdtab} applies to each $F_i$ and we conclude that $[F_i]\in\II(\II^{rich}_2)$, so the same is true about $[F]$.
\end{proof}

\begin{proof}[Proof of Theorem~\ref{thm:mingens3}]
 From now on we will use without mention the fact that if $T$ is a Young tabloid containing repeated entries in some column then $T=0$. In particular, the only Young$^{\d}_3$ tabloids that will concern us have columns of size at most $3$, and moreover the entries in each such column are $\{1,2,3\}$. We first show that if $T$ is a minimal generator of $\KK_{1,2}$, then it is (up to relabeling) one of the tabloids described in the statement of the theorem. Let $T$ be one such minimal generator. We are in one of the following situations:

\ul{\emph{Case 1: at least two columns of size $3$.}} Suppose that $T$ contains $2$ columns of size $3$ and some other column of size larger than $1$. We apply Lemma \ref{lem:three12s} to conclude that $T\in\II(2)$: the columns of size $3$ form an initial set $\mc{C}$, and we either have that $|\mc{C}|\geq 3$, or there exists a column $C$ of size $2$ which together with $\mc{C}$ forms an initial set; applying the said lemma to $\mc{C}$ in the former case, or to $\mc{C}\cup \{C\}$ in the latter, yields the conclusion. It follows that if $T$ contains (at least) two columns of size $3$, then the only way $T$ could be a minimal generator of $\KK_{1,2}$ is if one of the alternatives in (\ref{eq:two3}) occurs.
 
\ul{\emph{Case 2: one column of size $3$.}} Suppose now that $T$ has precisely one column $C$ of size $3$. If $T$ contains two columns $C_1,C_2$ of size $2$ with the same entries $\{x,y\}$, then up to a permutation of columns of size $2$, $\mc{C}=\{C,C_1,C_2\}$ forms an initial set of columns, thus Lemma \ref{lem:three12s} applies. It follows that if $T$ contains at least four columns of size $2$ then two of them will have the same entries, so $T=0$ in $\KK_{1,2}$.

If $T$ contains no column of size $2$, then $T=0$: switching the labels $1$ and $2$ permutes the columns of size $1$, and performs a transposition in the column of size $3$, so $T=-T$. If $T$ contains exactly one column of size $2$, then applying the shuffling relation of Lemma~\ref{lem:fundrels}(b) to that column yields $T=\pm T\pm T$ (for an appropriate choice of signs), so $T=0$.

Assume now that $T$ contains two columns of size $2$. Up to relabeling the entries, we may assume that they are $\ytableaushort{1,2}$ and $\ytableaushort{1,3}$. If the two columns belong to the same $T^j$, then switching $2$ and $3$ has the effect of permuting columns of $T^j$ of size $2$, and performing a transposition in the column of $T$ of size $3$. This gives $T=-T$, so $T=0$. We may then assume that the two columns lie in distinct $T^j$s, so the only possibilities are, up to relabeling, the ones depicted in (\ref{eq:one3two2}).

Finally, assume that $T$ contains three columns of size $2$: since no two contain the same entries, they must be $C_3=\ytableaushort{1,2}$, $C_2=\ytableaushort{1,3}$, and $C_1=\ytableaushort{2,3}$. For $i\in[3]$, consider $j$ such that $T^j$ contains $C_i$. If $T^j$ contains a column of size one with entry equal to $i$, then applying the shuffling relation to that entry and $C_i$ expresses $T=T'+T''$, where each of $T',T'',$ contains two columns of size two with the same entries. As explained previously, this implies that $T',T''\in\II(2)$, so $T=0$ in $\KK_{1,2}$. Combining the fact that every $i\in[3]$ appears in $T^j$ the same number (namely $d_j$) of times with the fact that no $T^j$ containing $C_i$ can contain a column of size one with entry $i$ leaves us with the possibilities described in (\ref{eq:one3three2}).

\ul{\emph{Case 3: no column of size $3$.}} Suppose now that $T$ has only columns of size at most two and write $\ll$ for its shape. Since $\II_{\ll}=\II^{rich}_{\ll}$ we may assume that $T$ has at least $4$ columns of size $2$, so that the graph $G$ associated to $T$ is a rich graph. Throughout the analysis of this case we will freely go back and forth between graphs and tabloids. If three edges of $G$ join the same pair of vertices, then $[G]\in\II(2)$ by Proposition~\ref{prop:idealgraphs} (or the argument from \emph{Case 1}). We may thus assume that $4\leq|E(G)|\leq 6$.

If $|E(G)|=6$ then $E(G)$ must contain two of each of the edges $E_i$ joining the vertices in $[3]\setminus\{i\}$. It follows by the same argument as in the last paragraph of \emph{Case 2} that for each edge $E_i$ of color $j$, the vertex $\ccirc{i}$ must be $j$--saturated, so the only possibilities are the ones depicted in (\ref{eq:zero3six2}).

If $|E(G)|=5$ then we may assume that $G$ contains two of each of $E_2,E_3$, and one $E_1$. Moreover, $\ccirc{1}$ must be $j$--saturated, where $j$ denotes the color of $E_1$. It follows that $T$ is as depicted in (\ref{eq:zero3five2}), or it is (up to relabeling and column permutations) one of
\[\ytableausetup{boxsize=1.25em,tabloids,aligntableaux=center}
  \ytableaushort{11211,23323}\quad\textrm{ or }\quad\ytableaushort{112,233}\oo\ytableaushort{11,23}.\]
In both these cases, switching $2\leftrightarrow 3$ permutes the columns of size two within each $T^s$, and performs a transposition in the column $\ytableaushort{2,3}$, thus $T=-T$ and $T=0$.

Finally, suppose that $|E(G)|=4$. We may assume that $\ll$ is such that $\ll^1_2\geq\ll^2_2\geq\cdots$. We consider the columns $C_1,C_2,C_3,C_4$ of $T$, ordered from left to right. After possibly using the shuffling relation of Lemma~\ref{lem:fundrels}(b), we may assume that $C_1$ and $C_2$ don't have the same entries, so up to relabeling they are $C_1=\ytableaushort{1,2}$ and $C_2=\ytableaushort{1,3}$. Using again the shuffling relation, we may assume that $C_3=\ytableaushort{1,i}$ for $i=2$ or $3$: the only situation when the shuffling relation could not be applied is if $C_1,C_2$ and $C_3=\ytableaushort{2,3}$ were all contained in $T^1$ and $d_1=2$, but in this case we could use the shuffling relation on $C_4$ and the (skew-)symmetries of $T$ in order to conclude that $T=\pm T\pm T$, hence $T=0$. Using the shuffling relation once more on $C_4$, we may assume that $C_4\neq\ytableaushort{2,3}$, and moreover $C_4\neq C_3$, since otherwise $T$ would contain $3$ columns with the same entries, hence $T\in\II(2)
$. It follows that we may assume that $C_3=C_1,C_4=C_2$ or $C_3=C_2,C_4=C_1$.

If both $C_1$ and $C_2$ belong to $T^1$ then after interchanging them and switching the labels $2\leftrightarrow 3$ throughout $T$, we may assume that $C_3=C_1$ and $C_4=C_2$, so we are either in one of the situations described in (\ref{eq:zero3four2grps}), or $C_3$ is also a column of $T^1$. We show that this can't be the case (we use $(\oo)$ to indicate the possibility of $C_4$ being a column of $T^1$, and $\equiv$ to indicate congruence modulo $\II(2)$):
\[\ytableausetup{boxsize=1.25em,tabloids,aligntableaux=center}
T=\ytableaushort{111,232}(\oo)\ytableaushort{1,3}\overset{\rm{Lemma~\ref{lem:fundrels}(b)}}{=}\ytableaushort{111,232}(\oo)\ytableaushort{2,3}+\overbrace{\ytableaushort{111,232}(\oo)\ytableaushort{1,2}}^{\in\II(2)}\]
\[\ytableausetup{boxsize=1.25em,tabloids,aligntableaux=center}
\overset{\rm{Lemma~\ref{lem:fundrels}(b)}}{\equiv}\ytableaushort{111,233}(\oo)\ytableaushort{2,3}+\ytableaushort{113,232}(\oo)\ytableaushort{2,3}
\]
For the first term, we have
\[\ytableausetup{boxsize=1.25em,tabloids,aligntableaux=center}
 \ytableaushort{111,233}(\oo)\ytableaushort{2,3}\overset{\rm{Lemma~\ref{lem:fundrels}(b)}}{=}-\ytableaushort{111,233}(\oo)\ytableaushort{3,2}\overset{2\leftrightarrow 3}{=}-\ytableaushort{111,322}(\oo)\ytableaushort{2,3}=-T,
\]
while the second term is zero by the argument in the preceding paragraph. We get $2T\in\II(2)$, so $T=0$ in $\KK_{1,2}$.

If $C_2$ is not in $T^1$, then $\ll^1_2=1$, so $\ll^j_2=1$ for $j=2,3,4$, i.e. $T$ is one of the tabloids in~(\ref{eq:zero3four2sings}).


\emph{Minimality of the tabloids in (\ref{eq:two3})--(\ref{eq:zero3six2}).} So far we have proved that the tabloids in (\ref{eq:two3})--(\ref{eq:zero3six2}) generate $\KK_{1,2}$. Conversely, we must show that these tabloids are non-zero (satisfy no relation) modulo $\II(2)$. From the description of $\II^{rich}_2$ it follows easily that if $\ll\vdash^n 3\d$, $\II(2)_{\ll}$ is spanned by tabloids that have at least four columns containing the entries $\{1,2\}$, in particular $\ll$ must have at least four columns of size larger than $1$. 

This shows that $\II(2)_{\ll}=0$ for all the shapes $\ll$ appearing in (\ref{eq:two3})--(\ref{eq:one3three2}) and (\ref{eq:zero3five2})--(\ref{eq:zero3six2}). It remains to show that the corresponding tabloids are non-zero (which is a simple, but tedious computation), or alternatively, given the fact that $m_{\ll}(\pi(\SSS))=0$, it is enough to prove that $m_{\ll}(\SSS)=1$ (this can be done using the plethysm formulas from \cite[Exercise~9(b), Section~I.8]{macdonald}). We leave the details to the interested reader (see below for an example).

It remains to deal with the tabloids in (\ref{eq:zero3four2grps})--(\ref{eq:zero3four2sings}). In all cases, we have $m_{\ll}(\II(2))=1$, where $\rm{hwt}_{\ll}(\II(2))$ is generated by the tabloid whose columns of size $2$ are $\ytableaushort{1,2}$. Since $m_{\ll}(\pi(\SSS))=0$, it suffices to show that $m_{\ll}(\SSS)=2$ for the shapes in (\ref{eq:zero3four2grps}), and $m_{\ll}(\SSS)=3$ for the shape in (\ref{eq:zero3four2sings}). We treat the shape in (\ref{eq:zero3four2sings}), and leave the others to the interested reader. 

Translating back into the $\GL$--setting, we need to show that if $\ll$ is the $n$--partition with $\ll^j=(3d_j-1,1)$ for $j=1,2,3,4,$ and $\ll^j=(3d_j)$ for $j>4$, then the multiplicity of the Schur functor $S_{\ll}$ in $\Sym^3(\Sym^{d_1}\oo\cdots\oo\Sym^{d_n})$ is equal to $3$. If $\nu\vdash 3$, then it follows from \cite[Exercise~9(b), Section~I.8]{macdonald} that for $j\leq 4$, $S_{\nu}\circ\Sym^{d_j}$ contains $S_{\ll^j}$ if and only if $\nu=(2,1)$, while for $j>4$, $S_{\nu}\circ\Sym^{d_j}$ contains $S_{\ll^j}$ if and only if $\nu=(3)$, in both cases the multiplicity of $S_{\ll^j}$ being equal to $1$. If we write $\chi_{\nu}$ for the character of the irreducible representation of $\SS_3$ associated to $\nu$, then since $\chi_{(3)}$ is the trivial character we get (see \cite[Example~2.6]{ful-har} for the character table of $\SS_3$)
\[m_{\ll}(\SSS)=\scpr{\chi_{(2,1)}^4\cdot\chi_{(3)}^{n-4}}{\chi_{(3)}}=\scpr{\chi_{(2,1)}^4}{\chi_{(3)}}\]
\[=\frac{1}{|\SS_3|}\cdot(2^4+0\cdot(\#\rm{transpositions})+(-1)^4\cdot(\#3\rm{-cycles}))=3.\]
\end{proof}

\subsection{Minimal generators of degree $4$}\label{subsec:mingens4}
 In this section we describe the minimal generators of degree $4$ of the generic ideal $\II$ of $\t(X)$.
\begin{theorem}\label{thm:mingens4}
 The syzygy functor $\KK_{1,3}$ is non-zero if and only the multiset $\{d_1,\cdots,d_n\}$ contains one of $\{3\}$, $\{2,1\}$ or $\{1,1,1\}$. A minimal list of generators of $\KK_{1,3}$ is then given (up to a permutation of $d_1,\cdots,d_n$, and under certain specified restrictions) as follows:
\begin{itemize}
 \item When $d_1=3$
\[\ytableausetup{boxsize=1.25em,tabloids,aligntableaux=center}
  \ytableaushort{111223,233444}.
\]
 \item When $n\geq 2$, $d_1=2$ and $d_2=1$
\begin{equation}\label{eq:42}
\ytableausetup{boxsize=1.25em,tabloids,aligntableaux=center}
  \ytableaushort{1122,3344}\oo\ytableaushort{13,24}.
\end{equation}
 \item When $n\geq 3$, and $d_1=d_2=d_3=1$
\[\ytableausetup{boxsize=1.25em,tabloids,aligntableaux=center}
  \ytableaushort{12,34}\oo\ytableaushort{12,34}\oo\ytableaushort{13,24}.
\]
\end{itemize}
\end{theorem}

\begin{remark}
 Note that in all cases depicted above, the underlying graph (without the coloring of the edges) is the same as the one in Example~\ref{ex:LW4graph}.
\end{remark}

\begin{proof}[Proof of Theorem~\ref{thm:mingens4}]
 Since the ideal of the subspace variety (Example \ref{ex:subspace}) is generated in degree $3$, the only $\ll$'s for which $m_{\ll}(\KK_{1,3})$ could be non-zero are those for which each $\ll^j$ has at most two parts. For such $\ll$, $\II_{\ll}=\II^{rich}_{\ll}$ so we may take the minimal tabloids $T$ that generate $\KK_{1,3}$ to be of the form $[G]$ for $G$ a rich graph on $4$ vertices. In particular, $T$ has at least five columns of size $2$. Furthermore, since the ideal of the secant variety $\s_2(X)$ is generated in degree $3$, it follows that graphs $G$ containing a triangle give rise to elements of $\II(3)$, thus $[G]=0$ in $\KK_{1,3}$. We conclude by Proposition~\ref{prop:MCB}(2) that $\KK_{1,3}$ is generated by tabloids $T=[G]$, where $G$ is an MCB--graph. Consider such a $T$ and the corresponding $G$.

If $G$ has type $(3,1)$ then since it has at least $5$ edges, it must contain a rich subgraph with $3$ vertices, so $[G]\in\II(3)$ by Proposition~\ref{prop:idealgraphs}. It follows that $G$ must have type $(2,2)$, so by Proposition~\ref{prop:MCB}(4), it must contain an even number of edges, i.e. at least $6$. If $G$ contains at least $8$ edges, then it contains a rich subgraph with at most $3$ vertices, so $[G]\in\II(3)$ again by Proposition~\ref{prop:idealgraphs}. We may thus assume that $G$ is an MCB--graph of type $(2,2)$ with $6$ vertices. Since every vertex of $G$ is connected to (at most) two other, if some vertex of $G$ has $4$ incident edges then they determine a rich subgraph with at most $3$ vertices, i.e. $[G]\in\II(3)$ as before. We may then assume that each vertex of $G$ has degree at most $3$, but then it follows that the degree has to be exactly $3$ because $G$ has $6$ edges and type $(2,2)$. The only possibility for $G$ is then (up to relabeling the vertices and coloring and orienting the 
edges), with $\{\ccirc{1},\ccirc{4}\}$ and $\{\ccirc{2},\ccirc{3}\}$ forming the sets in the bipartition.
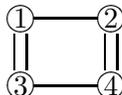
\[
\xymatrixrowsep{.3pc}
\xymatrix{
\ccircle{1} \ar@{-}[dd]<-.5ex> \ar@{-}[dd]<.5ex> \ar@{-}[r]  & \ccircle{2} \ar@{-}[dd]<-.5ex> \ar@{-}[dd]<.5ex> & \\
 &  & \\
\ccircle{3} \ar@{-}[r] & \ccircle{4} & \\
}\]
If $G$ contains an edge $E$ of color $c$ and a vertex $\ccirc{x}$ not incident to $E$ which is not $c$--saturated, then writing the shuffling relation with $E$ and $x$ expresses $[G]=[G_1]+[G_2]$, where each $G_i$ contains a rich subgraph with at most $3$ vertices, i.e. $[G]\in\II(3)$, a contradiction. Assume now that $\ccirc{x}$ is $c$--saturated, and consider $\ccirc{y}$ the vertex of $E$ which is adjacent to $\ccirc{x}$. Writing $e_c$ for the number of edges of color $c$ joining $\ccirc{x}$ and $\ccirc{y}$, we get that $\ccirc{y}$ is incident to $e_c+1$ such edges, so $d_c\geq e_c+1$, which forces $\ccirc{x}$ to be incident to at least one other edge of color $c$. Taking $E=(1,2)$ and $\ccirc{x}=\ccirc{3}$ shows that $(1,2)$ and $(3,4)$ must have the same color, and either the remaining edges have the same color, or they split up into two pairs of opposite edges of the same color. It follows easily that $T=[G]$ is as described in the statement of the theorem, or
\[\ytableausetup{boxsize=1.25em,tabloids,aligntableaux=center}
  T=\ytableaushort{1312,2434}\oo\ytableaushort{12,34}\overset{\rm{Lemma~\ref{lem:fundrels}(b)}}{=}\overbrace{\ytableaushort{1321,2434}\oo\ytableaushort{12,34}}^{\in\II(3)}+\ytableaushort{1313,2424}\oo\ytableaushort{12,34},
\]
where the first term is in $\II(3)$ because the corresponding graph contains a triangle, while the second term coincides with the one in (\ref{eq:42}) after switching $2\leftrightarrow 3$ and permuting the columns in $T^1$.

We conclude as in the proof of Theorem~\ref{thm:mingens3}, either by showing that each tabloid $T$ we found is non-zero via a direct calculation, or by showing that for the given shapes $\ll$, $m_{\ll}(\SSS)=1$ (for this, we used the SchurRings package \cite{SchurRings} to compute the plethysms $\Sym^4(\Sym^3)$, $\Sym^4(\Sym^2\oo\Sym^1)$, and $\Sym^4(\Sym^1\oo\Sym^1\oo\Sym^1)$).
\end{proof}

\section*{Acknowledgments}
 We would like to thank J.M. Landsberg, Giorgio Ottaviani, and Bernd Sturmfels for useful suggestions. The Macaulay2 algebra software \cite{M2} was helpful in many experiments, particularly through the SchurRings package \cite{SchurRings} which was used to predict some of the syzygy functors described in the paper.


\end{document}